\documentclass[a4paper, 11 pt, twoside]{amsart}
\usepackage{bwieneck}[2011/03/03]
\usepackage[paper = a4paper, left = 3.4cm, right = 3.4cm, headsep = 6mm,
footskip = 10mm, top = 34mm, bottom = 34mm, footnotesep=5mm, headheight =
2cm]{geometry}

\usepackage{color}
\usepackage[T1]{fontenc}
\usepackage{lmodern} 
\usepackage{mathrsfs}
\usepackage{textcomp}
\usepackage{amssymb}
\usepackage{comment}
\usepackage{mathtools} 
\usepackage{tikz}

\usepackage{multibib} 

\normalsize

\usepackage[hypertexnames=false,backref=page,pdftex,
	breaklinks=true,
	extension=pdf,
	colorlinks=true,
	linkcolor=blue,
	citecolor=blue,
	urlcolor=blue,
	pdfpagemode=NonOutlines
	]{hyperref}

\renewcommand{\epsilon}{\varepsilon}

\title[]{Monodromy Invariants and Polarization Types of Generalized Kummer Fibrations}

\author{Benjamin Wieneck}
\address{Benjamin Wieneck\\ Institut für Algebraische Geometrie\\ Leibniz Universit\"at Hannover\\Welfengarten 1\\30167 Hannover\\Germany}
\email{bwieneck@gmail.com}

\begin{document}

\thanks{Mathematics Subject Classification 32J27, 14D06, 32Q15, 53D12, 53C26, 32G13, 14D20}
\begin{abstract} In this paper a monodromy invariant for isotropic classes on generalized Kummer type manifolds is constructed. This invariant is used to determine the polarization type of Lagrangian fibrations on such manifolds -  a notion which was introduced in an earlier paper of the author. The result shows that the polarization type of a Lagrangian fibration of generalized Kummer type depends on the connected component of the moduli space. 
\end{abstract}

\maketitle
\tableofcontents

\section{Introduction} \label{intro}

In this paper we continue our study of the polarization type of Lagrangian fibrations on irreducible holomorphic symplectic manifolds which we started in \cite{benni1}.

A \emph{Langrangian fibration} $f : X \rightarrow B$ is a holomorphic map from an irreducible holomorphic symplectic manifold $X$ to a normal complex space of dimension $\frac{1}{2} \dim X$ with connected fibers such that the restriction of the holomorphic symplectic form on $X$ to the regular part of each fiber of $f$ vanishes. It is well known that all smooth fibers are abelian varieties even if $X$ is not projective. Given a smooth fiber $F$ an immediate question is to ask for natural polarizations on it which is by definition the first Chern class $H = c_1(L)$ of an ample line bundle $L$ of $F$. 

It is known that for each smooth fiber $F$ one can find a Kähler class $\omega$ on $X$ such that the restriction $\omega|_F$ is integral and primitive and hence defines a polarization on $F$, see Proposition \ref{specialpol}. An ad--hoc definition of the \emph{polarization type} of a Lagrangian fibration would be to set $\pol(f) := \pol(\omega|_F)$ where the latter one is the polarization type of the polarization on $F$ given by $\omega|_F$. It follows that this does not depend on the chosen $F$ and $\omega$ and that the polarization type stays constant in families of Lagrangian fibrations.  For a summary see also section \ref{poltypes}.

In \cite{benni1} we proved that the polarization type of Lagrangian fibrations of $\kdrei$--type is always principal. This motivates the speculation whether the polarization type only depends on the deformation class of the irreducible holomorphic symplectic manifold $X$ itself, forgetting the Lagrangian fibration. 

The purpose of this paper is to show that this is not the case, a fact which came as a surprise to us. Indeed, the following holds. 

\begin{theo}[Theorem \ref{mainpol}, Proposition \ref{bmkummer}]\label{mainkum} Let $f : X \rightarrow \IP^n$ be a Lagrangian fibration of generalized Kummer type. If $d = \Div(\lambda)$ denotes the divisibility\footnote{Here we mean with the divisibility $k = \Div(\lambda)$, the largest positive number $k$, such that $(\lambda, \cdot)/k$ is an integral form.} of $\lambda = c_1(f^\star \ko_{\IP^n}(1))$, then $d^2$ divides $n+1$ and we have for the polarization type
	$$ \pol(f) \ = \ \left(1, \ldots, 1, d, \frac{n+1}{d} \right) \, .$$
	Furthermore, for a fixed dimension $\dim X = 2n$, the divisibilities of classes $\lambda$ as above which can appear for the generalized Kummer type, are exactly the positive integers $d$ such that $d^2$ divides $n+1$. 
\end{theo}

The proof of Theorem \ref{mainkum} involves moduli theory of Lagrangian fibrations of generalized Kummer type, as for instance exploited in \cite{eyal2}. The moduli theory appears in form of what is called a \emph{monodromy invariant}. 

Let $X$ be an irreducible holomorphic symplectic manifold and consider the monodromy group $\Mon^2(X)$. A \emph{faithful monodromy invariant}, see section \ref{moninviso} and \cite[Def.~5.16]{eyalprime}, is a $\Mon^2(X)$--invariant map $\vartheta : I(X) \rightarrow \Sigma$  where $I(X) \subset H^2(X,\IZ)$ is a $\Mon^2(X)$--invariant subset and $\Sigma$ is an arbitrary set, such that the induced map $I(X)/\Mon^2(X) \rightarrow \Sigma$ is injective.

The following is a generalized Kummer analogue of E. Markman's monodromy invariant for the $\kdrei$ case, see \cite[2.]{eyal2}.

Let $X$ be of generalized Kummer type. For a fixed positive integer $d$, let us denote $I_d(X) \subset H^2(X,\IZ)$ the set of all primitive isotropic classes with divisibility $d$. For the case that $d^2$ divides $n + 1$, let $\Sigma_{n,d}$ denote the set of isometry classes of pairs $(H,w)$ such that $H$ is a lattice isometric to the lattice $L_{n,d}$ which is defined in \refb{latticelnd} and $w \in H$ is a primitive class with $(w,w) = 2n + 2$.

\begin{theo}[Theorem \ref{moninv}]\label{mainmoninv} Let $X$ be a generalized Kummer type manifold of dimension $2n$ and $d$ a positive integer such that $d^2$ divides $n+1$. There is a surjective faithful monodromy invariant $$\vartheta \ : \ I_d(X) \longrightarrow \Sigma_{n,d}$$ of the manifold $X$. 
\end{theo}

A similar result as Theorem \ref{mainmoninv} was obtained independently by G.\,Mongardi und G.\,Pacienza in \cite{monpac}. They also construct a faithful monodromy invariant function, see \cite[Lem.\,3.4]{monpac}, which requires a choice of an embedding of $H^2(X,\IZ)$  into the Mukai lattice, cf. \ref{mukailattice}. Also compare with \cite[Rem.\,3.10]{monpac} and Theorem \ref{orbit}. \\

\textbf{Structure of the paper.} In section \ref{hyperk} we give a review of the theory of hyperkähler manifolds. Section \ref{poltypes} is a summary of the author's paper \cite{benni1} about the definition of the polarization type of a Lagrangian fibration.
In section \ref{anorbit} a canonical orbit of primitive isometric embeddings from the generalized Kummer lattice into the Mukai lattice (of torus type) is constructed which is a main ingredient for the construction of the monodromy invariant which is done in the next section \ref{moninviso}.
Section \ref{bmgenkum} has the purpose to recall the construction of Beauville--Mukai systems of generalized Kummer type and to determine their polarization types. An excursion to the theory of Jacobians is needed, see subsection \ref{casejacobian}. Finally, we compute the polarization type of a Lagrangian fibration of generalized Kummer type in section \ref{computation}. \\

\textbf{Acknowledgements.} I thank my advisor Klaus Hulek, Eyal Markman and Giovanni Mongardi for helpful discussions. I thank Olivier Debarre for pointing me out the idea of the proof of Lemma \ref{multiple}.

\section{Hyperkähler Manifolds and their fibrations} \label{hyperk}

In this section we recall the basic facts about irreducible holomorphic symplectic manifolds and their fibrations which are Lagrangian.

\begin{defin}A compact Kähler manifold $X$ is called \emph{hyperkähler} or \emph{irreducible holomorphic symplectic} if $X$ is simply connected and $H^0(X, \Omega_X^2)$ is generated by a nowhere degenerate holomorphic two--form $\sigma$. \end{defin} Note that $\sigma$ is automatically symplectic since every holomorphic form on a compact Kähler manifold is closed.

The most basic example is provided by the Douady space $S^{[n]}$ of $n$ points for a K$3$ surface $S$ which parametrizes zero--dimensional subspaces of $S$ of length $n$. A. Beauville \cite{Bea84} showed that $S^{[n]}$ is an irreducible holomorphic symplectic manifold of dimension $2n$. 

In this paper we are more interested in the following example. We start with a complex two--torus $S$ and take the the Douady space $S^{[n+1]}$. This is holomorphic symplectic, but not simply connected. Then one uses the \emph{Douady--Barlet map}
$$ \rho \ : \ S^{[k]} \longrightarrow S^{(k)}\, , \ \ \ Z \longmapsto \sum\limits_{z \in Z} (\dim_{\CC} \ko_{Z,z}) z $$
which is a resolution of singularities of the \emph{symmetric product} $S^{(k)} := (S \times \cdots \times S)/\Sigma_{k}$ to obtain a morphism 
$$ S^{[n+1]} \stackrel{\rho}{\longrightarrow} S^{(n+1)} \stackrel{+}{\longrightarrow} S $$
where the last map is summation in $S$. By A.~Beauville \cite{Bea84} the fiber $S^{[[n]]} = K_n(S)$ over $0$ is an irreducible holomorphic symplectic manifold of dimension $2n$, called \emph{generalized Kummer manifold} as $S^{[[1]]}$ is the usual Kummer K$3$ surface. 

An irreducible holomorphic symplectic manifold is of \emph{$\kdrei$--type} or of \emph{generalized Kummer type} if it is deformation equivalent to $S^{[n]}$ for a K$3$ surface $S$ or to $S^{[[n]]}$ for a two--torus $S$, respectively. \\

The second cohomology $H^2(X,\IZ)$ of any irreducible holomorphic symplectic manifold $X$ admits the well known \emph{Beauville--Bogomolov--Fujiki} quadratic form $q_X$ which is non--degenerate and of signature $(3, b_2(X) - 3)$, see \cite[23.3]{huy1}. The associated bilinear form is denoted by $(\cdot, \cdot)$. On an abstract lattice we also denote the bilinear form by $(\cdot, \cdot)$. The lattice $H^2(X,\IZ)$ with the Beauville--Bogomolov--Fujiki form is a deformation invariant of the manifold $X$. For manifolds of generalized Kummer type this lattice is isometric to the abstract \emph{generalized Kummer lattice}
$$ \Lambda \ = \ \UU^{\oplus 3} \oplus \left\langle -(2+2n) \right\rangle \, ,$$
see \cite[Prop.~8]{Bea84} where $\left\langle -(2+2n)\right\rangle$ denotes the lattice of rank one with generator $l$ such that $(l,l) = -(2+2n)$ and $\UU$ the unimodular rank two hyperbolic lattice.

A \emph{marking} on an irreducible holomorphic manifold $X$ is the choice of an isometry $\eta : H^2(X,\IZ) \rightarrow \Lambda$. The pair $(X,\eta)$ is then called a \emph{marked pair} or a \emph{marked irreducible holomorphic symplectic manifold}. \\

If $X$ is a fixed irreducible holomorphic symplectic manifold, set $\Lambda := H^2(X,\IZ)$ and consider the Kuranishi family $\pi : \gX \rightarrow \Def(X)$ with $\gX_0 := \pi^{-1}(0) = X$. We will view the base $\Def(X)$ sometimes as a germ but also as a representative which we choose small enough i.e. simply connected. Then by Ehresmann's theorem, we can choose a trivialization $\Sigma : R^2 \pi_{\star}\IZ \rightarrow \Lambda_{\Def(X)}$ also called a \emph{marking} and define the \emph{local period map} by
$$ \kp : \Def(X) \longrightarrow \IP(\Lambda_{\CC}) \, , \ \ t \longmapsto [\Sigma_t(H^{2,0}(\gX_t))]$$
where $\Lambda_{\CC} := \Lambda \otimes \CC$. It takes values in the \emph{period domain of type $\Lambda$} \cite[22.3]{huy1} namely
$$\Omega_{\Lambda} \ := \ \left\{ p \in \IP(\Lambda_{\CC}) \ | \ (p,p) = 0 \text{ and } (p, \bar{p}) > 0 \right\} $$
which is connected since the signature of $q_X$ is $(3, \rk \Lambda - 3)$.
\begin{theo}[\textsc{Local Torelli},  \cite{Bea84}, 8.] The period map $\kp : \Def(X) \rightarrow \Omega_{\Lambda}$ is an open embedding. \end{theo}
Two marked pairs $(X_i, \eta_i)$, $i=1,2$, are called \emph{isomorphic} if there is an isomorphism $f : X_1 \rightarrow X_2$ such that $\eta_2 = \eta_1 \circ f^{\star}$.
There exists a \emph{moduli space of marked pairs} $\gM_{\Lambda} := \left\{ (X,\eta) \text{ marked pair } \right\}/\cong$ which can be constructed by gluing all deformation spaces $\Def(X)$ of irreducible holomorphic symplectic manifolds $X$ with $H^2(X,\IZ)$ isometric to $\Lambda$. This gives a non--Hausdorff complex manifold of dimension $\rk \Lambda - 2$. The \emph{global period mapping}  
$$ \kp \ : \ \gM_{\Lambda} \longrightarrow \Omega_{\Lambda} \, , \ \ \ (X,\eta) \longmapsto [\eta(H^{2,0}(X))] $$ 
is locally given by $\kp : \Def(X) \rightarrow \Omega_{\Lambda}$ and hence is again a local biholomorphism by the Local Torelli. If one takes an arbitrary connected component $\gM^{\circ}_{\Lambda}$ of $\gM_{\Lambda}$ then by a result of D. Huybrechts \cite[Prop. 25.12]{huy1} the restriction $\kp : \gM^{\circ}_{\Lambda} \rightarrow \Omega_{\Lambda}$ is surjective. 

If $L$ denotes a line bundle on $X$ by abuse of notation we also denote the universal family of the pair $(X,L)$ by $\pi : \gX_L \rightarrow \Def(X,L)$ which comes with an universal line bundle $\kl$ on $\gX_L$ such that $(\gX_L)_0 = X$ and $\kl_0 = L$, see \cite[Cor. 1]{Bea84}. We consider again $\Def(X,L)$ as a germ but as well as a proper space. A representative of $\Def(X,L)$ is locally given by $(c_1(L), \cdot) = 0$ in $\Omega_{\Lambda}$ hence it is a smooth hypersurface in $\Def(X)$, see \cite[26.1]{huy1} and one defines $\gX_L$ as the preimage of it under $\pi$. The family $\pi : \gX_L \rightarrow \Def(X,L)$ is the restriction of the Kuranishi family $\pi : \gX \rightarrow \Def(X)$ to $\gX_L$ and $\Def(X,L)$. \\

\begin{ueber}\textbf{Lagrangian fibrations.} Due to D. Matsushita much is known about nontrivial fiber structures on irreducible holomorphic symplectic manifolds.
	\begin{theo}[\textsc{Matsushita}, \cite{Mat99}, \cite{Mat00}, \cite{Mat01}, \cite{Mat03}]\label{mats} Let $f : X \rightarrow B$ be a surjective holomorphic map with connected fibers from an irreducible holomorphic symplectic manifold $X$ of dimension $2n$ to a normal complex space $B$ such that $0 < \dim B < 2n$. Then the following statements hold.
		\begin{enumerate} 
			\item $B$ is projective of dimension $n$ and its Picard number is $\rho(B) = 1$.
			\item For all $t \in B$, the fiber $X_t \coloneqq f^{-1}(t)$ is a Lagrangian subspace \ie $\sigma|_{X^{\text{\emph{reg}}}_t} = 0$ where $X^{\text{\emph{reg}}}_t$ denotes the smooth part of $X_t$.
			\item If $X_t$ is smooth then it is a projective complex torus \ie an abelian variety. \end{enumerate}
	\end{theo} 
	Such a fibration $f : X \rightarrow B$ as in the Theorem is called a \emph{Lagrangian fibration}. If $X$ is a $\kdrei$--type manifold then we call $f : X \rightarrow B$ a \emph{$\kdrei$--type fibration}. \\
	
	If the base of the Lagrangian fibration is smooth even more is known due to a deep result of J.-M. Hwang which was recently slighty generalized by C. Lehn and D. Greb to the non--projective case.
	\begin{theo}[\textsc{Hwang}, \cite{Hwa08}, \cite{LeGr13}] Let $f : X \rightarrow B$ be a Lagrangian fibration such that $B$ is smooth and $\dim X = 2n$. Then $B \cong \IP^n$. \end{theo} If $f : X \rightarrow B$ is a $\kdrei$--type fibration then E. Markman \cite[Thm. 1.3, Rem. 1.8]{eyal1} in combination with a result of D. Matsushita \cite[Thm. 1.2, Cor. 1.1]{matiso} has shown that $B \cong \IP^n$ without assuming smoothness of $B$. By \cite[Appendix]{yoshibase} also in combination with \cite[Thm. 1.2, Cor. 1.1]{matiso} this holds for Lagrangian fibrations of generalized Kummer type. 

The basic example of a Lagrangian fibration on a generalized Kummer manifold can be obtained as follows. Let $f : S \rightarrow E$ be a surjective holomorphic map where $S$ is a two torus and $E$ is an elliptic curve.  With the of the Douady--Barlet map we have a map
$$S^{[[n]]} \hookrightarrow S^{[n+1]} \stackrel{\rho}{\longrightarrow} S^{(n+1)} \stackrel{f \times \cdots \times f}{\longrightarrow} E^{(n+1)} \cong \IP^n \times E \, .$$
This map and the projection from $\IP^n \times E$ to $\IP^n$ defines a Lagrangian fibration $S^{[[n]]} \rightarrow \IP^n$ by Matsushita's Theorem \ref{mats}. Let $F$ denote a smooth fiber of $p$, then the fiber of the Lagrangian fibration $S^{[[n]]} \rightarrow \IP^n$ is isomorphic to the abelian subvariety of $F^{n+1}$ given by the equation $x_1 + \ldots + x_{n+1} = 0$ for $(x_1, \ldots, x_{n+1}) \in F^{n+1}$. 

Note that two--dimensional Lagrangian fibrations are exactly the elliptic K$3$ surfaces.

\begin{defin}\label{deflf} \begin{enumerate} \item A \emph{family of Lagrangian fibrations} over a connected complex space $S$ with finitely many irreducible components is an $S$--morphism 
		$$   \xymatrix{  \kX \ar[rr]^{\phi} \ar[dr] & & P \ar[dl] \\ 
			&	S & 		   } $$
		where $\kX \rightarrow S$ is a family of irreducible holomorphic symplectic manifolds and $P \rightarrow S$ is a family of projective varieties such that for every $s \in S$ the restriction $\phi|_{\kX_s} : \kX_s \rightarrow P_s$ to the irreducible homorphic symplectic manifold $\kX_s$ is a Lagrangian fibration. 
		\item Two Lagrangian fibrations $f_1$ and $f_2$ are \emph{deformation equivalent} if there is a family of Lagrangian fibrations $\phi$ over a connected complex space $S$ containing $f_1$ and $f_2$ \ie there are points $t_i \in S$ such that $\phi_{t_i} = f_i$, $i = 1,2$.
	\end{enumerate}
\end{defin}

\begin{defin}\cite[5.2]{eyalprime}\label{pairdef} Let $X_i$, $i = 1, 2$, denote two irreducible holomorphic symplectic manifolds, $L_i$ holomorphic line bundles on $X_i$ and $e_i$ classes in $H^2(X_i,\IZ)$. \begin{enumerate}
		\item The pairs $(X_1,e_1)$ and $(X_2, e_2)$ are called \emph{deformation equivalent} if there exists a familty $\pi : \kX \rightarrow S$ of irreducible holomorphic symplectic manifolds over a connected complex space $S$ with finitely many irreducible components, a section $e$ of $R^2 \pi_{\star} \IZ$, points $t_i$ in $S$ such that $\kX_{t_i} = X_i$ and $e_{t_i} = e_i$.
		\item The pairs $(X_1, L_1)$ and $(X_2, L_2)$ are called \emph{deformation equivalent} if there exists a family $\pi : \kX \rightarrow S$ of irreducible holomorphic symplectic manifolds over a connected complex space $S$ with finitely many irreducible components, a line bundle $\kl$ on $\kX$, points $t_i$ in $S$ such that $\kX_{t_i} = X_i$ and $\kl_{\kX_{t_i}} = L_i$. 
	\end{enumerate}
\end{defin}

\begin{rem} Note that we can reformulate (ii) of Definition \ref{pairdef} as the following. \begin{itemize}
		\item The pairs $(X_1, L_1)$ and $(X_2, L_2)$ are called \emph{deformation equivalent} if there exists a family $\pi : \kX \rightarrow S$ of irreducible holomorphic symplectic manifolds over a connected complex space $S$ with finitely many irreducible components, a section $e$ of $R^2\pi_{\star} \IZ$ which is everywhere of Hodge type $(1,1)$, points $t_i$ in $S$ such that $\kX_{t_i} = X_i$ and $e_{t_i} = c_1(L_i)$.
	\end{itemize}
	Clearly, $e_t \coloneqq c_1(\kl_t)$ would give such a section. Conversely, given a section $e$ as in the alternative definition, we get a line bundle $L_t$ on $\kx_t$ corresponding to $e_t \in  H^{1,1}(\kx_t, \IZ)$ with respect to the isomorphism $\Pic(\kx_t) \cong H^{1,1}(\kx_t, \IZ)$ since $\kx_t$ is irreducible holomorphic symplectic. Then the Kuranishi family of the pair $(\kx_t, L_t)$ gives an universal line bundle on the respective total space for every $t \in S$. Those line bundles glue to a line bundle $\kl$ on $\kx$ with the property $c_1(\kl_t) = e_t$.\end{rem}

		\begin{prop}\label{defogen} Let $f_i : X_i \rightarrow \IP^n$, $i = 1,2$, denote two Lagrangian fibrations of generalized Kummer or $\kdrei$--type and set $L_i \coloneqq f_i^{\star}\ko_{\IP^n}(1)$. The Lagrangian fibrations $f_i$ are deformation equivalent in sense of \emph{Definition \ref{deflf}}, if and only if the pairs $(X_i, L_i)$ are deformation equivalent. 	\end{prop} 
		\begin{proof} The proof is exactly the same as in \cite[Prop. 3.9]{benni1}. There everything is stated for the $\kdrei$--type, but it carries over to the generalized Kummer type word by word. 	\end{proof}

\begin{lem}\cite[Lem.~3.5]{benni1} \label{liso}Let $f : X \rightarrow B$ be a Lagrangian fibration and let $L \coloneqq f^{\star}A$ be the pullback of a line bundle $A$ on $B$. \begin{enumerate}
		\item $L$ is isotropic with respect to the Beauville--Bogomolov quadratic form. 
		\item If $A$ admits nontrivial sections then $L$ is nef. 
		\item If $X$ is of $\kdrei$ or generalized Kummer type and $A$ is primitive, then $L$ is primitive.
	\end{enumerate}\end{lem}
	\begin{proof} The first two statements are contained in \cite[Lem.~3.5]{benni1}. The third statement is formulated for the $\kdrei$--type, but the proof works also in the generalized Kummer case with use of \cite[Cor.~1.1]{matiso} and $b_2(X) \geq 3$.	\end{proof}

\end{ueber}

\begin{ueber}\textbf{Orientation.} We summarize section 4.~of \cite{eyal1}.

Let $b_2 > 0$ a positive integer and $\Lambda$ be an even lattice of signature $(3, b_2 - 3)$. Define
$$\widetilde{\kc}_{\Lambda} \ := \ \left\{ x \in \Lambda_{\IR} \ | \ (x,x) > 0 \right\} \, .$$
We have the following.
\begin{lem}\label{wcone} \cite[Lem. 4.1]{eyal1} If $W \subset \Lambda_{\IR}$ is a three dimensional subspace such that the bilinear form of $\Lambda$ is positive definite on it, then $W \setminus \{ 0 \}$ is a deformation retract of $\ctilde_{\Lambda}$. Therefore $H^2(\ctilde_{\Lambda},\IZ) \cong \IZ$ is a free abelian group of rank one.  The reflection $R_u$ for $u \in \Lambda$ with $(u,u) \neq 0$ given by $$ R_u(x) \ := \ (x,x) - 2 \frac{(u,x)}{(u,u)} u \, , $$ acts on $H^2(\widetilde{\kc}_{\Lambda},\IZ)$ \begin{itemize} \item as $+1$  if $(e,e) < 0$ and \item as $-1$ if $(e,e) > 0$, \end{itemize} therefore it defines a generator of $H^2(\widetilde{\kc}_{\Lambda},\IZ)$. \end{lem}

In particular, the Lemma implies that $\ctilde_{\Lambda}$ is connected, as $H_0(\ctilde_{\Lambda}, \IZ) = H_0(W \setminus \{ 0 \}, \IZ) = \IZ$.

\begin{defin} \label{orientationclass}
	An \emph{orientation of $\ctilde_{\Lambda}$} is a choice of a generator of $H^2(\widetilde{\kc}_{\Lambda},\IZ) \cong \IZ$.
\end{defin}

By speaking of oriented isometries of the lattice $\Lambda$, we mean isometries which preserve the orientation of $\ctilde_{\Lambda}$ in sense of the definition above: every isometry $g : \Lambda \rightarrow \Lambda$ induces a homeomorphism $g : \ctilde_{\Lambda} \rightarrow \ctilde_{\Lambda}$, therefore we have a morphism 
\begin{align}\label{spinornorm}
\OO(\Lambda) & \longrightarrow \Aut(H^2(\ctilde_{\Lambda}, \IZ)) \cong \{ \pm 1 \} \\
g & \longmapsto g^{\star} \, . \notag
\end{align}

\begin{defin}\label{orientpreserve} The morphism in \refb{spinornorm} above is also called \emph{spinor norm}. Its kernel is denoted by $\OO^+(\Lambda)$ and isometries in it are called \emph{orientation preserving}. \end{defin}

For a period $p \in \Omega_{\Lambda}$ let $\Lambda(p)$ denote the integral Hodge structure of weight two of $\Lambda$ determined by the period $p$, that is
\begin{equation}\label{hsperiod} \Lambda^{2,0}(p) = p \, , \ \ \ \Lambda^{0,2}(p) = \bar{p} \ \ \text{ and } \ \ \Lambda^{1,1}(p) = \left\{ x \in \Lambda_{\CC} \ | \ (x,p) = (x, \bar{p}) = 0\right\} \, .\end{equation}
As in the geometric situation, we also set $$\Lambda^{1,1}(p, R) \ \coloneqq \ \left\{ x \in \Lambda_R \ | \ (x,p) = 0 \right\}$$ for $R \in \left\{ \IZ , \IR \right\}$. Furthermore, consider the set
\begin{equation} \label{cpzwei}  \kc'_p \coloneqq \left\{ x \in \Lambda^{1,1}(p, \IR) \ | \ (x,x) > 0 \right\}  .\end{equation} 
The restriction of the bilinear form to $\lambone(p, \IZ)$ has signature $(1, b_2 - 3)$. 
 Therefore $\kc'_p$ has two connected components.

Let $x$ be in $\kc'_p$ with $p = \CC\cdot \sigma$. We can define a subspace 
\begin{equation} W_x \ \coloneqq \ \re(p) \oplus \imm(p) \oplus \IR \cdot x \end{equation}
of $\Lambda_{\IR}$ such that the bilinear form is positive definite on it. The subspace $W_x$ of $\Lambda_{\IR}$ defines a generator of $H^2(\ctilde_{\Lambda}, \IZ) \cong \IZ$ in the following way. 

The subvector space $W_x$ has the canonical ordered basis
\begin{equation}\label{wbasis}
(\re(\sigma), \imm(\sigma), x) \, ,
\end{equation}
which defines an orientation in the ordinary sense \ie a volume form $\beta(\sigma) \coloneqq \re(\sigma)^\star \wedge \imm(\sigma)^\star \wedge x^\star$  of the manifold $W_x \setminus \{ 0 \}$. The orientation $\beta(\sigma)$ does not depend on the choice of $\sigma$, indeed we have $\beta(\lambda \sigma) = |\lambda| \beta(\sigma)$ for any $\lambda \in \CC$. Take the two sphere $\IS^2 \subset W_x \setminus \{ 0 \}$ in $W_x$. It is well known, that the basis \refb{wbasis} gives a volume form on $\IS^2$ by restricting the two form
$$ x_1 \im(\sigma)^\star \wedge x^\star + x_2 x^\star \wedge \re(\sigma)^\star + x_3 \re(\sigma)^\star \wedge \im(\sigma)^\star    $$
to $\IS^2$, where $x_1, x_2, x_3$ are the standard coordinates with respect to the basis \refb{wbasis}. Use
\begin{equation}\label{obtainorientation}
H^2(\IS^2, \IZ) = H^2(W_x \setminus \{ 0 \}, \IZ) = H^2(\ctilde_{\Lambda}, \IZ)\end{equation} to obtain a generator of  $H^2(\ctilde_{\Lambda}, \IZ)$ \ie an orientation in sense of Definition \ref{orientationclass}. Obviously we end up with the other generator, if we change the orientation of $W_x$ given by the basis \refb{wbasis}.

\begin{prin} \label{principle2} An element $x$ in $\kc'_p$ for a period $p \in \Omega_{\Lambda}$ determines a generator of $H^2( \ctilde_{\Lambda}, \IZ) \cong \IZ$ \ie an orientation of $\ctilde_{\Lambda}$. The two generators are distinguished by the two connected components of $\kc'_p$. Therefore a connected component of $\kc'_p$ determines an orientation of $\ctilde_{\Lambda}$.
\end{prin}

\end{ueber}

\begin{ueber}\textbf{The geometric situation.}\label{geomsit} Let $\gM_{\Lambda}$ denote the moduli space of isomorphism classes of marked pairs $(X, \eta)$ of type $\Lambda$ \ie $X$ is an irreducible holomorphic symplectic manifold and $\eta : H^2(X,\IZ) \rightarrow \Lambda$ is a marking. Choose a connected component $\gM^{\circ}_{\Lambda}$ of $\gM_{\Lambda}$. 
	Recall that for $(X,\eta) \in \gM^{\circ}_{\Lambda}$ there is a canonical choice for the connected component of $$ \kc'_X \ \coloneqq \ \left\{ x \in H^{1,1}(X,\IR) \ | \ (x,x) > 0 \right\}$$ namely the \emph{positive cone} $\kc_X$ which contains the Kähler cone $\kk_X$ of $X$. Therefore, by Principle \ref{principle2} 
	$$ \ctilde_X \ \coloneqq \ \ctilde_{H^2(X,\IZ)} \ = \ \{ x \in H^2(X, \IR) \ | \ (x,x) > 0 \} $$ 
	has a \emph{natural orientation}, which determines an orientation in sense of Definition \ref{orientationclass} of $\ctilde_{\Lambda}$ via the homeomorphism  $\eta : \ctilde_X \cong \ctilde_{\Lambda}$. 
	
	\begin{defin}\label{compatible} We will refer to the orientation of $\tilde{\kc}_{\Lambda}$ (in sense of Definition \ref{orientationclass}) which is induced by the marking $\eta$ and the natural orientation of $\ctilde_X$ for some (hence for all) marked pair $(X,\eta)$ in $\gM^{\circ}_{\Lambda}$, as the orientation \emph{compatible} to the connected component $\gM^{\circ}_{\Lambda}$ of the moduli of marked pairs.   \end{defin}

	Consider the period map 
	$$ \kp \ : \ \gM^{\circ}_{\Lambda} \longrightarrow \Omega_{\Lambda} \, , \ \ \ (X,\eta) \longmapsto [\eta(H^{2,0}(X))] $$ 
	and set $p \coloneqq \kp(X,\eta)$. Then  $\eta(H^{1,1}(X,\IR)) = \lambone(p,\IR)$. An orientation of $\ctilde_{\Lambda}$ determines a connected component \begin{equation}
	\label{cp} \kc_p \subset \kc'_p \end{equation} of $\kc'_p$ by Principle \ref{principle2}. Equivalently, we can characterize the orientation compatible to $\gM^{\circ}_{\Lambda}$ by the condition $\eta(\kc_X) = \kc_p$ for all $(X,\eta) \in \gM^{\circ}_{\Lambda}$ with $p = \kp(X,\eta)$.
\end{ueber}

\begin{ueber}\textbf{$\Omega^+_{\lambda^\bot}$\label{foriso} for an isotropic class.} For the following see also \cite[4.3]{eyal2}. We are still in the setting of \ref{geomsit}. Let $\lambda \in \Lambda$ be a nontrivial isotropic class. We define a hyperplane section 
	\begin{equation}\label{hypersectioniso} \Omega_{\lambda^{\bot}} \ \coloneqq \ \Omega_{\Lambda} \cap \lambda^{\bot} \ = \ \{ p \in \Omega_{\Lambda} \ | \ (p, u) = 0  \} \, . \end{equation}
	Note that the bilinear form on $\lambda^\bot \subset \Lambda_{\IR}$ is degenerate since $\lambda$ is isotropic. The hyperplane section $\Omega_{\lambda^{\bot}}$ has two connected components and we can still obtain a natural connected component of it from the geometrical situation in the following way. 
	
	For $p \in \Omega_{\lambda^\bot}$, $\lambda$ belongs to $\Lambda^{1,1}(p,\IR)$ and is contained in the boundary of one of the connected components of $\kc'_p$ since $\lambda$ is isotropic. For $(X,\eta) \in \gM^{\circ}_{\Lambda}$, either $\eta^{-1}(\lambda)$ or $\eta^{-1}(-\lambda)$ belongs to $\partial \kc_X$. We assume that the former is the case, otherwise take $-\lambda$. Then consider only periods $p$ in $\Omega_{\lambda^\bot}$ such that $\lambda$ belongs to the closure of the distinguished connected component $\kc_p$ in $\Lambda^{1,1}(p,\IR)$, see \refb{cp}, determined by the orientation of $\ctilde_{\Lambda}$ compatible to $\gM^{\circ}_{\Lambda}$ \ie
	\begin{equation}
	\label{hyperisocon}  \Omega^+_{\lambda^\bot} \ \coloneqq \ \left\{ p \in \Omega_{\lambda^\bot} \ | \ \lambda \in \partial \kc_p \right\} \end{equation}
	which is one of the connected components of $\Omega_{\lambda^{\bot}}$. Note that the only common element of the closures of the connected components of $\Omega_{\lambda^{\bot}}$ is the null vector, therefore $\Omega^+_{\lambda^\bot}$ of \refb{hyperisocon} is indeed one of the connected components of $\Omega_{\lambda^{\bot}}$. We refer to $\Omega^+_{\lambda^\bot}$ as the \emph{compatible} connected component of $\Omega_{\lambda^{\bot}}$ with respect to the chosen connected component $\gM^\circ_\Lambda$ of the moduli of marked pairs. \end{ueber}

\begin{ueber}\textbf{Monodromy.} We recall some basic definitions and state G.~Mongardi's monodromy result \cite{mon1}.
	 \begin{defin} Let $X_i$, $i = 1,2$, be two irreducible holomorphic symplectic manifolds. An isometry $P : H^2(X_1, \IZ) \rightarrow H^2(X_2, \IZ)$ is called a \emph{parallel transport operator} if there exists a family $\pi : \kx \rightarrow S$ of irreducible holomorphic symplectic manifolds, points $t_i$ such that $\kx_{t_i} = X_i$ and a continuous path $\gamma$ such that the parallel transport $P_{\gamma}$ along $\gamma$ in the local system $R^2 \pi_{\star} \IZ$ coincides with $P$. For $X := X_1 = X_2$ it is also called a \emph{monodromy operator} and the subgroup $\Mon^2(X)$ of $O(H^2(X,\IZ))$ generated by monodromy operators is called the \emph{monodromy group}. \end{defin}

Let $\Lambda$ denote a non--degenerate lattice of signature $(3,b_2 -3)$.
\begin{defin}\label{wx} Let $\kw(\Lambda)$ denote the subgroup of $\OO^+(\Lambda)$ consisting of orientation preserving isometries acting as $\pm 1$ on the discriminant $\Lambda^{\vee}/\Lambda$. Denote by $$\chi : \kw(\Lambda) \rightarrow \left\{ \pm 1 \right\}$$ the associated character. We also write $\kw(X) := \kw(H^2(X,\IZ))$
	for an irreducible holomorphic manifold $X$.	\end{defin}

For a class $u \in \Lambda$ with $(u,u) \neq 0$ we have the rational reflection $R_u : \Lambda \rightarrow \Lambda$ defined by
\begin{equation}
\label{rationalreflection}
R_u(x) \ \coloneqq \ x - 2\frac{(u,x)}{(u,u)} u  \, .\end{equation} 
If $(u,u) < 0$, then by Lemma \ref{wcone} the reflection $R_u$ is orientation preserving in sense of Definition \ref{orientpreserve} \ie contained in $\OO^+(\Lambda_{\IQ})$.

\begin{defin}\label{orientreflection} Let $\Lambda$ be a non--degenerate lattice of signature $(3, b_2 - 3)$. For a class $u \in \Lambda$ with $(u,u) \neq 0$, denote $\rho_u : \Lambda_{\IQ} \rightarrow \Lambda_{\IQ} \in \OO^+(\Lambda_{\IQ})$ the orientation preserving  isometry defined by
	$$	\rho_u \ \coloneqq \  \begin{cases} \phantom{-}R_u & \text{if } (u,u) < 0 \, , \\
	-R_u & \text{if } (u,u) > 0 \, .
	\end{cases} $$	 	\end{defin}

\begin{rem}\label{determinant} \begin{enumerate}
		\item If $(u,u) = \pm 2$, then $R_u$ and $\rho_u$ define honest integral isometries $\Lambda \rightarrow \Lambda$. 
		
		\item The action of $R_u$ on $\Lambda^\vee$ for a $h \in \Lambda^\vee$ is $$R_u(h)(x) \ = \ h(R_u(x)) \ = \ h(x) - (2 \frac{h(u)}{(u,u)} u , x)  ,$$
		\ie $R_u(h) =  h \mod \Lambda$, hence for $(u,u) = \pm 2$ the isometry $\rho_u$ is contained in $\kw(\Lambda)$. More precisely we have
		$$	\chi(\rho_u) \ = \  \begin{cases} +1 & \text{if } (u,u) < 0 \, , \\
		-1 & \text{if } (u,u) > 0 \, . \end{cases}$$
		\item The isometry $R_u$ satisfies $R_u(u) = -u$ and $R_u|_{u^{\bot}} = \id_{u^{\bot}}$, hence we have for the determinant $\det(R_u) = -1$. Therefore
		$$	\det(\rho_u) \ = \  \begin{cases} -1 & \text{if } (u,u) < 0 \, , \\
		(-1)^{b_2 + 1} & \text{if } (u,u) > 0 \, .
		\end{cases} $$
		Note that for the $\kdrei$ and generalized Kummer case $b_2$ is odd and for the O'Grady examples $b_2$ is even.
	\end{enumerate}\end{rem}
	
	\begin{theo}[\textsc{Mongardi}, {\cite[Thm. 2.3]{mon1}}]\label{giovanni} Let $X$ be a generalized Kummer $n$--type manifold. Then $\Mon^2(X)$ consists precisely of orientation preserving isometries $g \in \kw(X)$ such that $\chi(g) \cdot \det(g) = 1$. \end{theo}
	
	In particular, for a generalized Kummer manifold $X$, $\Mon^2(X)$ is an index $2$ sub group of $\kw(X)$ as $|\kw(X)/\Mon^2(X)| = |\im( \det \cdot \chi)| = 2$. 
	
	\begin{cor}\label{nevercontained} For a generalized Kummer type manifold $X$, the monodromy group $\Mon^2(X)$ is an index $2$ sub group of $\kw(X)$. The orientation preserving isometry $\rho_u \in \kw(X)$ for a class $u \in H^2(X,\IZ)$ with $(u,u) = \pm 2$ defined in \emph{Definition \ref{orientreflection}} is never contained in $\Mon^2(X)$.	
	\end{cor}
	\begin{proof}
		The first statement we have just discussed. The second statement follows from Remark \ref{determinant} (ii) and (iii). \end{proof}
\end{ueber}

\section{Polarization types of Lagrangian fibrations}\label{poltypes}
	
	The author introduced the following notion in \cite{benni1}. Let $f : X \rightarrow B$ be a Lagrangian fibration. We known that all smooth fibers are abelian varieties by Theorem \ref{mats}, even if $X$ is not projective.
For an abelian variety $F$ of dimension $\dim F = n$, there is a well known classical notion of a \emph{polarization}, cf. \cite[p. 70]{BL}, which is by definition the first Chern class $H = c_1(L)$ of an ample line bundle $L$ of $F$. Often one calls the ample line bundle $L$ a polarization. Furthermore, one can associate to such a polarization a \emph{type}, which is a tuple $$\pol(L) \ = \ (d_1,\ldots,d_n)$$ of positive integers such that $d_i$ divides $d_{i+1}$, cf. \cite[p. 70]{BL}. 

Given a smooth fiber $F$ of the Lagrangian fibration $f$ we want to consider polarization on it induced from $X$. First of all, it is not clear, how to obtain a polarization on a smooth fiber $F$ of the Lagrangian fibration $f : X \rightarrow B$ if $X$ is not projective. However, due to the following statement, which is related to an observation of C. Voisin \cite[Prop. 2.1]{camp05}, it is always possible.

\begin{prop}\label{specialpol}\cite[Prop~4.3]{benni1} For any smooth fiber $F$ there is a K\"ahler class $\omega$ on $X$ such that the restriction $\omega|_F$ is integral and primitive. \end{prop}

Such a class $\omega$ is called \emph{special K\"ahler class} (with respect to $F$) and defines a polarization $\omega|_F$ on the abelian variety $F$ in the sense above. To this polarization one can associate its type $\pol(\omega|_F)   \coloneqq (d_1, \ldots, d_n)$ where again $d_i$ are positive integers such that $d_i$ divides $d_{i+1}$.

\begin{defin} The \emph{polarization type} of a Lagrangian fibration $f : X \rightarrow B$ is $$\pol(f) \ \coloneqq \ \pol(\omega|_F) \ = \ (d_1, \ldots, d_n) \, .$$ \end{defin}

This definition seems to be a bit ad--hoc, but it is convenient for a short introduction. The following statements were shown in \cite{benni1}.

\begin{theo}\label{mainpolar} \cite[Section 4]{benni1} Let $f : X \rightarrow B$ be a Lagrangian fibration with $\dim X = 2n$. Then the following statements hold. \begin{enumerate}
		\item \cite[Prop.~4.7]{benni1} The polarization type $\pol(f)$ is well defined \ie does not depend on the chosen smooth fiber and the chosen special Kähler class (with respect to this fiber) and is a primitive vector in $\IZ^n$. 
		\item \cite[Thm.~4.9]{benni1} The polarization type is a deformation invariant of the fibration \ie if $f' \colon X' \rightarrow B'$ is a Lagrangian fibration deformation equivalent to $f$, then $\pol(f) = \pol(f')$.
		\item \cite[Prop.~4.6, Prop.~4.10]{benni1} Let $B^\circ$ denote the subset of $B$ which para\-metrizes the smooth fibers. Then there exists a \emph{family of special Kähler classes}, that is a map $\alpha : B^\circ \rightarrow \kh$ where $\kh \subset (R^2 \pi_{\star} \IZ \otimes \ko_B)|_{B^\circ}$ is a subbundle and $\alpha(t)$ is a special Kähler class with respect to the smooth fiber $X_t$ for every $t \in B^\circ$. In particular $\pol(\alpha(t)) = \pol(f)$ for every $t \in B^\circ$. 
		\item \cite[Prop.~4.6]{benni1} The family of special Kähler classes $\alpha$ induces a holomorphic map, called \emph{moduli map},
		\begin{align*}
		\phi \ : \  B^{\circ} & \longrightarrow \ka_{\pol(f)} \, , \\
		t & \longmapsto (X_t, \alpha(t)) \notag \end{align*}
		where $\ka_{\pol(f)}$ denotes the moduli space of $\pol(f)$ polarized abelian varieties. 
		\item \cite[Thm.~6.1]{benni1} Let $f : X \rightarrow \IP^n$ be a Lagrangian fibration of $\kdrei$--type. Then $\pol(f) = (1, \ldots, 1)$.
	\end{enumerate}
\end{theo}

In this paper, we want to determine the polarization type of a Lagrangian fibration of generalized Kummer type.

\section{An orbit of primitive isometric embeddings}\label{anorbit}

The main ingredient for the construction of a monodromy invariant for isotropic classes in the second cohomology of a generalized Kummer manifold is a monodromy invariant orbit of primitive isometric embeddings of the Kummer--type lattice into the Mukai lattice. \\

The group of isometries $\OO(\mukai)$ of the \emph{Mukai lattice} $\mukai := \Lambda \oplus \UU$ (see also below) and $\OO(\Lambda)$ acts on the set $\OO(\Lambda,\mukai)$ of primitive isometric embeddings $\iota : \Lambda \hookrightarrow \mukai$ of the lattice $\Lambda$ into $\mukai$ by composition i.e. for $g \in \OO(\Lambda)$ and $\tilde{g} \in \OO(\mukai)$ one sets $g \cdot \iota := \iota \circ g$ and $\tilde{g} \cdot \iota := \tilde{g} \circ \iota$.

\begin{defin} Let $\iota \in \OO(\Lambda,\mukai)$ be a primitive isometric embedding. An element $g \in \OO(\Lambda)$ \emph{leaves the $\OO(\mukai)$--orbit $[\iota] = \OO(\mukai) \iota$ invariant} if $g \cdot [\iota] := [\iota \circ g] = [\iota]$ i.e. if there exists $\tilde{g} \in \OO(\mukai)$ such that $\tilde{g} \circ \iota = \iota \circ g$. The orbit is called \emph{monodromy invariant} if  $\Mon^2(X) \cdot [\iota] = [\iota]$ i.e. all elements in $\Mon^2(X)$ leave the orbit $[\iota]$ invariant. \end{defin}

\begin{rem}\label{vtov}  Let $\iota : \Lambda \hookrightarrow \mukai$ denote a primitive isometric embedding. If $X$ is a generalized Kummer type manifold then $\iota(\Lambda)^{\bot} = \left\langle v\right\rangle$ is of rank $1$ since the Mukai lattice is of rank $8$ and the Kummer type lattice is of rank $7$. An isometry $\tilde{g} \in \OO(\mukai)$ with $\iota \circ g = \tilde{g} \circ \iota$ necessarily satisfies $\tilde{g}(\iota(\Lambda)) = \iota(\Lambda)$ and $\tilde{g}(v) = \pm v$, otherwise $\tilde{g}$ cannot be an isometry. 
\end{rem}

The following Lemma is a special case of \cite[Cor. 1.5.2]{nikulin}.

\begin{lem}\label{extended} Let $\Lambda$ be the generalized Kummer or $\kdrei$ lattice. Write $\Lambda = w^{\bot} \subset \mukai$ with $w$ primitive (cf. \emph{Remark \ref{vtov}}). An isometry $g \in \OO(\Lambda)$ can be extended to an isometry $\tilde{g} \in \OO(\mukai)$ if and only if $g$ acts as $\pm 1$ on the discriminant $\Lambda^{\vee}/\Lambda$.\end{lem}
\begin{proof} By \cite[Cor.~1.5.2]{nikulin} we can extend $g$ to such a $\tilde{g}$ if and only if we have an isometry $\varphi : \Lambda^{\bot} \rightarrow \Lambda^{\bot}$ with an additional property. Since $\Lambda^{\bot} = \langle w \rangle$ the only two isometries are $\varphi = \pm 1$. Following the exposition in \cite[5.~ff.]{nikulin}, the additional property for $\varphi = \pm 1$ means that $g$ acts on $\Lambda^\vee / \Lambda$ as $\pm 1$.  \end{proof}

\begin{cor}\label{stabil} Let $\Lambda = w^{\bot} \subset \mukai$ be as in the Lemma above and let us denote  $[\iota] = \OO(\mukai) \iota$ an arbitrary invariant $\OO(\mukai)$--orbit of primitive isometric embeddings $\Lambda \hookrightarrow \mukai$. Then the sub group $\kw(\Lambda) \subset \OO^+(\Lambda)$ defined in \emph{Definition \ref{wx}} is equal to the  sub group of all $g \in \OO^+(\Lambda)$ leaving the orbit $[\iota] = \OO(\mukai) \iota$  invariant, i.e. there exists $\tilde{g}$ such that $\iota \circ g = \tilde{g} \circ \iota$. \end{cor} 
\begin{proof}An element $g \in \OO^+(\Lambda)$ leaves $\OO(\mukai) \iota$ invariant if and only if it acts by $\pm 1$ on the discriminant $\Lambda^{\vee}/\Lambda$ by Lemma \ref{extended}. 
\end{proof}
In other words, $\kw(\Lambda) = \Stab([\iota])$ is equal to the stabilizer of $[\iota]$ with respect to the action of $\OO^+(\Lambda)$ on the set of $\OO(\mukai)$--orbits of primitive isometric embeddings $\OO(\Lambda, \mukai)$. 


With the knowledge of the monodromy group of a generalized Kummer manifold, see Theorem \ref{giovanni}, one can construct an analogue of the monodromy invariant $\OO(\mukai)$--orbit as in \cite[Thm.~1.10]{eyal3}. \\

Let $S$ be an abelian surface and let $H^{\bullet}(S)$ denote the even cohomology i.e. 
$$ H^{\bullet}(S) \ := \ H^0(S,\IZ) \oplus H^2(S, \IZ) \oplus H^4(S,\IZ)$$
together with the bilinear form defined by $(v,w) := (v_2,w_2) - \int_S (v_0 \wedge w_4 + v_4 \wedge w_o)$ where $(v_2, w_2) = \int_S v_2 \wedge w_2$ denotes the intersection form on $H^2(S,\IZ)$ and $v = v_0 + v_2 + v_4$ with $v_{i} \in H^i(S,\IZ)$ the decomposition in $H^{\bullet}(S)$ and similarly for $w$. This lattice is even, unimodular, of rank $8$ and isometric to the \emph{Mukai lattice}
\begin{equation}\label{mukailattice} \tilde{\Lambda} \ := \ \UU^{\oplus 4} \end{equation}
where $\UU$ is the unimodular rank two hyperbolic lattice. We identify $H^4(S,\IZ) = \IZ$ where we use the Poincare dual to a point as a generator and similarly $H^0(S,\IZ) = \IZ$ by taking the Poincare dual of $S$. 

\begin{defin} A \emph{Mukai vector} is a triple $v = (r,c,s)$ in $H^0(S,\IZ) \oplus H^{1,1}(S,\IZ) \oplus H^4(S,\IZ)$. It is called \emph{positive} if one of the following cases are satisfied \begin{enumerate} \item $r > 0$
		\item $r = 0$, $c$ is effective and $s \neq 0$
		\item $r = c = 0$ and $s < 0$
	\end{enumerate} \end{defin}
	Let $v$ be a primitive Mukai vector on $S$. An ample divisor $H$ on $S$ is called $v$--generic if every $H$--semistable sheaf is $H$--stable. For a coherent sheaf $F \in \coh(S)$ set $v(F) := \ch(F) \,$\footnote{Note that $v(F) = (\rk(F), c_1(F), c_1^2(F)/2 - c_2(F))$} which is a Mukai vector as easily verified. Choose a positive and primitive Mukai vector $v = (r,c,s)$ with $c \in \NS(S)$ and $(v,v) \geq 6$ together with a $v$--generic ample class $H$. 
	General results of S. Mukai \cite{mukais} imply that the moduli space $M_H(v)$ of $H$--stable sheaves $F$ with Mukai vector $v(F) = v$ is a projective holomorphic symplectic manifold but not irreducible. By \cite[Thm. 0.1]{yoshi} the Albanese torus of $M_H(v)$ is $S \times S^\vee$. Consider the Albanese map 
	$$ \alb_v : M_H(v) \longrightarrow S \times S^{\vee} $$
	and set $K_H(v) := \alb_v^{-1}(0,0)$. Then we have $\dim K_H(v) = (v,v) - 2 =: 2n$ and by K. Yoshioka \cite[Thm 0.2]{yoshi} this is an irreducible holomorphic symplectic manifold of Kummer type.
	
	 We have Mukai's homomorphism of Hodge structures
	 $$\Theta_v \ : \ v^{\bot} \longrightarrow H^2(M_H(v), \IZ) $$
	 which can be defined as follows. Choose a quasi--universal family of sheaves $\ke$ on $S$ of simplitude $\rho \in \IN$, cf. \cite[Thm.~A.5]{mukaibundle}. That is a family of sheaves $\ke \in \coh(S \times M_H(v))$ on $S$ parametrized by $M_H(v)$ (in particular, $\ke$ is flat over $M_H(v)$) and for every class $F \in M_H(v)$ one has $\ke_{[F]} = \ke|_{S \times \left\{F\right\}} \cong F^{\oplus \rho}$. Then set 
	\begin{equation}\label{mukaihom} \Theta_v(x) \ := \ \frac{1}{\rho} \left[(\pr_{M_H(v)})_{!}\left(( \ch(\ke) (\pr_{S})^{\star}(\sqrt{\td(S)} x^{\vee})\right) \right]_2 \end{equation}
	 where $x^{\vee} = -x_0 + x_2 + x_4$ for $x = x_0 + x_2 + x_4$ and $[ \cdot ]_2$ denotes the part in $H^2(S,\IZ)$. Note that $\sqrt{\td(S)} = 1$ for an abelian surface $S$. For the details see \cite[1.2]{yoshi}, \cite{ogradyweight}, \cite{mukaibundle} and \cite{mukais}. 
	
	By composing with the restriction map $r: H^2(M_H(v),\IZ) \rightarrow H^2(K_H(v),\IZ)$ we obtain a morphism
\begin{equation}\label{mukaihom2} \Theta_v \ : \ v^{\bot} \longrightarrow H^2(M_H(v),\IZ) \longrightarrow H^2(K_H(v),\IZ) \end{equation}
	which is an isometry of Hodge structures by \cite[Thm.~0.2]{yoshi} and which we also denote by $\Theta_v$ by abuse of notation.

	\begin{theo}\label{orbit} Let $X$ be a manifold of generalized Kummer type of dimension $2n \geq 4$. Then there exists a canonical monodromy invariant $\OO(\mukai)$--orbit $\iota_X$ of primitive isometric embeddings $\Lambda = H^2(X,\IZ) \hookrightarrow \mukai$ into the Mukai lattice. \end{theo}
	\begin{proof} Let $K_H(v)$ denote the manifold of generalized Kummer type described above such that $\dim X = \dim K_H(v)$. Fix an isometry $\varphi : H^{\bullet}(S) \rightarrow \mukai$ and let $P : H^2(X,\IZ) \rightarrow H^2(K_H(v),\IZ)$ be a parallel transport operator. Denote by $\iota$ the primitive isometric embedding
		$$H^2(X,\IZ) \stackrel{P}{\longrightarrow} H^2(K_H(v),\IZ) \stackrel{\Theta_v^{-1}}{\longrightarrow} v^{\bot} \stackrel{\varphi}{\longrightarrow} \mukai \, .$$
		Set $\iota_X := \OO(\mukai) \iota$. Let $g \in \Mon^2(X)$ denote a monodromy operator. By Theorem \ref{giovanni} $g$ acts on $H^2(X,\IZ)^{\vee}/H^2(X,\IZ)$ as $\pm \id$. By Lemma \ref{extended} $g$ can be extended to an isometry $\tilde{g}$ of $\mukai$ such that $\iota \circ g = \tilde{g} \circ \iota$, i.e. the orbit $\iota_X$ is monodromy invariant. 
		
			The orbit $\iota_X$ is \emph{canonical} in the following sense. We have made a choice of moduli spaces $K_H(v) \subset M_H(v)$ of sheaves on an abelian surface $S$ and therefore of Mukai's homomorphism $\Theta_v : v^\bot \rightarrow H^2(M_H(v),\IZ) \rightarrow H^2(K_H(v),\IZ)$. It might be, that a different choice of moduli spaces and therefore of a different Mukai homomorphism could lead to another orbit of primitive isometric embeddings. With \emph{canonical} we mean that we always end up with the same orbit. 
			
			This follows from K.~Yoshioka's method of proof of the main results in \cite[4.3., Prop.~4.12., Proof~of~Thm.~0.1 and 0.2]{yoshi}. 
			If we choose another irreducible holomorphic symplectic moduli space of dimension $\dim X$, then it is deformation equivalent to $K_H(v)$ and Yoshioka's proof for this statement uses deformations of moduli spaces of sheaves over families of surfaces \cite[Lem.~2.3]{yoshi}, and Fourier--Mukai transforms for which the Mukai homomorphism varies continuously, see \cite[2.2., Proof of Prop.~2.4.]{yoshi}. Therefore the $\OO(\mukai)$--orbit does not change.
		\end{proof}
	
\section{Monodromy Invariants}\label{moninviso}

We start with basic facts about general monodromy invariants, as described in \cite[5.3.]{eyalprime}. In the next subsection the monodromy invariant for isotropic classes for the generalized Kummer case is constructed.

Let $X$ be an irreducible holomorphic symplectic manifold. Let $I(X) \subset H^2(X,\IZ)$ denote a monodromy invariant subset, i.e. $\Mon^2(X) \cdot I(X) \subset I(X)$ and $\Sigma$ a set.

\begin{defin}\cite[Def. 5.16]{eyalprime} A \emph{monodromy invariant} of the pair $(X,e)$, $e \in I(X)$, is a $\Mon^2(X)$--invariant map $\vartheta : I(X) \rightarrow \Sigma$ i.e. $\vartheta(ge) = \vartheta(e)$ for all $e \in I(X)$ and all $g \in \Mon^2(X)$. Further $\vartheta$ is called \emph{faithful} if the induced map $\bar{\vartheta } : I(X)/\Mon^2(X) \rightarrow \Sigma$ is injective. \end{defin}

\begin{ueber}\textbf{Induced monodromy invariant subset.} Let $X'$ denote another irreducible holomorphic symplectic manifold deformation equivalent to $X$. Let $P : H^2(X,\IZ) \rightarrow H^2(X',\IZ)$ denote a parallel transport operator. Then we can define
	$$I(X') \ := \ P(I(X))$$
	to obtain a $\Mon^2(X')$ invariant subset $I(X')$ of $H^2(X',\IZ)$ induced by $I(X)$. Indeed, this is well defined: if one has another parallel transport operator $P' : H^2(X,\IZ) \rightarrow H^2(X',\IZ)$, then $P'^{-1} \circ P$ is in $\Mon^2(X)$ hence $(P'^{-1} \circ P)(I(X)) = I(X)$ as $I(X)$ is $\Mon^2(X)$ invariant. Hence $P(I(X)) = P'(I(X))$. 
	
	Alternatively, we could define 
	$$I(X') \ = \ \left\{ e' \in H^2(X', \IZ) \ | \ \text{there exists } e \in I(X) \ \text{such that } (X,e) \sim_{\text{def}} (X',e') \right\} \, .$$
	where the deformation equivalence of the pairs $(X,e)$ and $(X,e')$ is meant in the sense of Definition \ref{pairdef} as usual. 
\end{ueber}

\begin{ueber}\textbf{Induced monodromy invariant.} Let $X'$ be as above. If we have a monodromy invariant $\vartheta : I(X) \rightarrow \Sigma$ then we can obtain an induced monodromy invariant on $X'$ which we also denote by $\vartheta : I(X') \rightarrow \Sigma$ by abuse of notation. If $e' \in I(X')$ then there is a pair $(X,e)$ deformation equivalent to $(X',e')$ and we can define the induced monodromy invariant by 
	$$\vartheta (e') \ := \ \vartheta(e) \, .$$
	Note that this is well defined as $\vartheta $ is $\Mon^2(X)$--invariant. \end{ueber}

The following is a very important statement for the computation of polarization types of Lagrangian fibrations and is based on the Global Torelli Theorem, see \cite[5.2 ff.]{eyalprime}.

\begin{prop}{\cite[Lem.~5.17]{eyalprime}}\label{computeinvariant} Let $\vartheta  : I(X) \rightarrow \Sigma$ be a faithful monodromy invariant and let $(X_i,e_i)$, $i = 1,2$, denote two pairs with $X_i$ deformation equivalent to $X$ and $e_i \in I(X_i)$. \begin{enumerate}
		\item $\vartheta (e_1) = \vartheta (e_2)$ if and only if $(X_1,e_1)$ and $(X_2,e_2)$ are deformation equivalent.
		\item If $\vartheta (e_1) = \vartheta (e_2)$ and $e_i = c_1(L_i)$ for holomorphic line bundles $L_i$ on $X_i$ and there exist Kähler classes $\omega_i$ on $X_i$ such that $(\omega_i, e_i) > 0$, then $(X_1,L_1)$ is deformation equivalent to $(X_2,L_2)$. \end{enumerate}
\end{prop}

For effective isotropic classes, the requirements of the second statement of the Proposition above is always satisfied due to the following Lemma.

\begin{lem}\label{isopos} \cite[Lem.~6.7]{benni1} Let $\lambda$ be a nontrivial isotropic class in the closure $\bar{\kc}_X$ of the positive cone in $H^{1,1}(X,\IR)$ with $X$ an arbitrary irreducible holomorphic symplectic manifold. Then the Beauville--Bogomolov quadratic form satisfies $(x, \lambda) > 0$ for every class $x$ in the positive cone $\kc_X$.\end{lem}
By definition the positive cone $\kc_X$ contains the Kähler cone $\kk_X$, therefore we always find Kähler classes as required in (ii) of Proposition \ref{computeinvariant}, if the considered classes $e_i$ are isotropic.




\begin{ueber}\textbf{Monodromy invariants for isotropic classes.} As one expects a close relation between Lagrangian fibrations and isotropic line bundles, similar for K$3$ surfaces, we are interested in monodromy invariants defined on the subset of isotropic classes of the second cohomology of an irreducible holomorphic symplectic manifold. In this section a monodromy invariant for the isotropic classes on generalized Kummer manifolds is constructed in analogy of \cite[2.]{eyal1} \\
	
	 Let $X$ be a generalized Kummer type manifold of dimension $2n$. By Theorem \ref{orbit} we have a canonical monodromy invariant $\OO(\mukai)$--orbit $\iota_X$ of primitive isometric embeddings from $\Lambda := H^2(X,\IZ)$ into the Mukai lattice $\mukai$ \refb{mukailattice}. 
  Choose the following:
\begin{enumerate}
	\item A representative $\iota : \Lambda \hookrightarrow \mukai$ in $\iota_X$.
	\item A generator $v$ of the sublattice $\iota(\Lambda)^{\bot} = \left\langle v\right\rangle$, cf. Remark \ref{vtov}.  \end{enumerate}

\begin{rem}\label{uniquetransitive} The Kummer type lattice $\Lambda$ has signature $(3,4)$, hence the orthogonal complement $\iota(\Lambda)^{\bot}$ is positive definite of rank one as the Mukai lattice $\mukai = \UU^{\oplus 4}$ has signature $(4,4)$. Since the Gram discriminant of $\Lambda$ is $-(2n+2)$ the Gram discriminant of $\iota(\Lambda)^{\bot}$ is $2n+2$, hence $(v,v) = 2n + 2$. 	
	Furthermore, by \cite[Thm 1.14.4]{nikulin} there is a unique orbit of such primitive elements with square $2n+2$ (respectively $2n-2$ in the $\kdrei$ case) in $\mukai$.  Since $\iota(\Lambda) = v^{\bot}$ we conclude that the action of $\OO(\Lambda) \times \OO(\mukai)$ on $\OO(\Lambda,\mukai)$ is transitive. 
\end{rem}

For a primitive and isotropic element $\alpha$ in the Kummer type lattice $\Lambda$ denote by $H(\alpha, \iota)$ the lattice defined by
\begin{equation}\label{hai} H(\alpha, \iota) \ := \ \sat\left\langle \iota(\alpha), v \right\rangle  \ = \ \sat\left\langle \iota(\alpha), -v \right\rangle \, ,\end{equation}
where $\sat$ denotes the saturation -- the \emph{saturation} of a sublattice $L$ is the maximal sublattice of the same rank containing $L$.

\begin{defin}\label{isometries} \begin{enumerate} \item Let $\Lambda_1, \Lambda_2$ denote lattices and $e_i \in \Lambda_i$ elements.  A \emph{morphism of the pairs} $(\Lambda_i, e_i)$ is an isometry $g : \Lambda_1 \rightarrow \Lambda_2$ such that $g(e_1) = e_2$.
\item The \emph{divisibility} or the \emph{divisor} of an element $x \in \Lambda$ is defined as $$ \Div(x) \ \coloneqq \ \max\left\{ k \in \IN \ | \ (x, \cdot)/k \text{ is an integral class in the dual } \Lambda^{\vee} \right\} \, . $$
Equivalently, $\Div(x)$ is the unique positive generator of the ideal $(x, L) = \Div(x) \IZ \subset \IZ$. Note that if the lattice is unimodular, then $\Div(x) = 1$ for every primitive element $x$. 		
\end{enumerate}		\end{defin}

Denote by 
\begin{equation}\label{vartheta} \vartheta(\alpha) \ := \ \left[ \left( H(\alpha, \iota), v \right)  \right] \end{equation}
the isometry class of the pair $(H(\alpha, \iota),v)$.

Let $d$ be a positive number such that $d^2$ divides $2n+2$. Then define the lattice $L_{n,d}$ as $\IZ^2$ with form
\begin{equation}\label{latticelnd} \frac{2n + 2}{d^2} \begin{pmatrix} 1 & 0 \\ 0 & 0 \end{pmatrix} \, . \end{equation}

	
	The following Lemma is very similar to \cite[Lem. 2.5]{eyal2}.
	
	\begin{lem}\label{welldefined} Let $\alpha \in \Lambda$ be a primitive isotropic class and set $d := \Div(\alpha)$.\begin{enumerate}
			\item $\vartheta(\alpha)$ does not depend on the chosen representative $\iota \in \iota_X$.
			\item For all $g \in \Mon^2(X)$ we have $\vartheta (g(\alpha)) = \vartheta(\alpha)$. 
			\item We can decompose $\alpha \in \Lambda \cong \UU^{\oplus 3} \oplus \left\langle -2n-2 \right\rangle$ as $$\alpha \ = \  d \xi + b \delta$$ where $\xi \in \UU^{\oplus 3}$ is primitive, $\delta$ is the generator of $\left\langle -2n-2\right\rangle$ and $\gcd(d,b) = 1$. Furthermore, $d^2$ divides $n+1$.
			\item The lattice $H(\alpha, \iota)$ is isometric to the lattice $L_{n,d}$ defined in \refb{latticelnd}.
			\item There is an integer $b$, namely the one in (iii), such that $(\iota(\alpha) - bv)/d$ is integral (\ie contained in $H(\alpha,\iota)$). Also any integer $b$ with 
			\begin{itemize} \item $\gcd(d,b) = 1$ and
				\item $(\iota(\alpha) - bv)/d$ is integral
			\end{itemize} satisfies $\vartheta(\alpha) = [(L_{n,d}, (d,b))]$. \end{enumerate}
	\end{lem}
	\begin{proof} \begin{enumerate} \item Let $\iota_i \in \iota_X$, $i = 1,2$, be two representatives with $\iota_i(\Lambda)^{\bot} = \left\langle  v_i \right\rangle$. Since the $\iota_i$ are in the same orbit $\iota_X$ there exists $\tilde{g} \in \OO(\mukai)$ such that $\tilde{g} \circ \iota_1 = \iota_2$ hence $\tilde{g}(\iota_1(\Lambda)) = \iota_2(\Lambda)$. We necessarily have $\tilde{g}(v_1) = \pm v_2$, otherwise we would have a contradiction to the bijectivity of $\tilde{g}$. We can assume $\tilde{g}(v_1) = v_2$ (otherwise take $-\tilde{g}$) then $\tilde{g}(\left\langle \iota_1(\alpha), v_1 \right\rangle) = \left\langle \iota_2(\alpha), v_2 \right\rangle$ and the same holds for the saturation. Consequently $\tilde{g}$ gives the desired isometry of the pairs $(H(\alpha, \iota_i), v_i)$ hence $\vartheta(\alpha)$ does not depend on the chosen $\iota$. 
			\item The orbit $\iota_X = \OO(\mukai) \iota$ is monodromy invariant that means we have a $\tilde{g} \in \OO(\mukai)$ such that $\tilde{g} \circ \iota = \iota \circ g$. With the same argument as in (i), we have $\tilde{g}(v) = \pm v$ (see Remark \ref{vtov}) and can assume $\tilde{g}(v) = v$. So $\tilde{g}$ defines an isometry between $\left\langle  \iota(\alpha) , v \right\rangle $ and $\left\langle  \iota(g(\alpha)) , v \right\rangle$ since $\tilde{g}(\iota(\alpha)) = \iota(g(\alpha))$ and in particular an isometry between the saturations $(H(\alpha, \iota), v)$ and $(H(g(\alpha), \iota),v)$, hence $\vartheta(\alpha) = \vartheta(g(\alpha))$. 
			
			\item Let $\delta$ be the generator of $\left\langle -2n-2\right\rangle \subset \Lambda$.  Then $\delta^{\bot}_{\Lambda} = \UU^{\oplus 3}$. Since $\alpha$ is primitive, we can write $\alpha = a \xi + b \delta$ such that $a > 0$ and $\xi \in \delta^{\bot} = \UU^{\oplus 3}$ and $\gcd(a,b) = 1$. Then
			$$ 0 \ = \ (\alpha, \alpha) \ = \ a^2(\xi,\xi) - (2n+2)b^2  \ \ \Leftrightarrow \ \ a^2 (\xi,\xi) \ = \ (2n+2)b^2 \, . $$
			As $(\xi,\xi)$ is even we get that $a^2$ divides $(n+1)$. Since $\delta$ is primitive we have $\Div(\delta) = 2n+2$ and $\Div(\xi) = 1$ as $\xi$ is primitive and $\UU^{\oplus 3}$ is unimodular, hence
			$$d \ = \ \Div(\alpha) \ = \ \gcd( \Div(a\xi), \Div(b\delta)) \ = \ \gcd(a, (2n+2)b) \ = \ a \, .$$ 
			\item[(\emph{iv}),(\emph{v})] We use the same notation as in (iii). The lattice $\iota(\UU^{\oplus 3})^{\bot} \subset \mukai$ is of rank $2$ and contains $\iota(\delta)$ and $v$, hence it is the saturation of $\left\langle \iota(\delta), v \right\rangle$ as orthogonal complements are always saturated. As a complement of a unimodular lattice it is unimodular itself, hence it is the hyperbolic plane $\UU$. Consequently, we can assume that $v = (1,n+1)$ and $\iota(\delta) = (1,-n-1)$. We have $\iota(\delta) - v = (2n+2)e$ where $e = (0,-1)$. Clearly $e$ is isotropic. Then set $$u \ := \ \frac{1}{d}(b v - \iota(\alpha)) \ = \ -\iota(\xi) - \frac{b}{d} (2n+2) e \, .$$
			Hence, the existence of such an integer $b$ is proven.
			
			
			As $\iota(\alpha) = -du +bv$ we have $\left\langle v, u \right\rangle \subset H(\alpha, \iota) := \sat \left\langle \iota(\alpha), v \right\rangle $. The complement $\delta^{\bot}_{\Lambda} = \UU^{\oplus 3}$ is unimodular, hence we can find $\eta \in \delta^{\bot}_{\Lambda}$ such that $(\eta, \xi) = 1$ as $\xi \in \UU^{\oplus 3}$ is primitive. For the intersection numbers we have
			$$ \begin{pmatrix} (v, e) & (v, \iota(\eta)) \\ (u, e) & (u, \iota(\eta)) \end{pmatrix} \ = \  \begin{pmatrix} -1 & \ 0 \\ \ 0 & -1 \end{pmatrix} \, . $$
			Therefore the sublattice $\left\langle v, u \right\rangle \subset \mukai$ must be saturated, otherwise the determinant of the matrix above must be divisible by a nontrivial square. Consequently we have
			$H(\alpha, \iota) := \sat \left\langle \iota(\alpha), v \right\rangle = \left\langle  v,u \right\rangle$.
			
			Furthermore, $(v,u) =  b\frac{2n+2}{d}$ and $(u,u) = b^2 \frac{2n+2}{d^2}$. The Gram matrix $G$ of $H(\alpha,\iota)$ with respect to the basis $v, u$ is therefore
			$$ G \ = \ \frac{2n+2}{d^2}  \begin{pmatrix} d^2 & bd \\ bd & b^2  \end{pmatrix}  \ = \   \frac{2n+2}{d^2}  \begin{pmatrix} d \\ b \end{pmatrix}   \begin{pmatrix} d & b  \end{pmatrix}  \, .$$
			Since $\gcd(d,b) = 1$ there are integers $i, j \in \IZ$ with $id + jb = 1$. Set
			$$ A \ := \ \begin{pmatrix} i & \phantom{-}j \\ b & -d \end{pmatrix} \, .$$
			This is an integral matrix with $A (d,b)^t = (1,0)^t$ and determinant $-1$, hence invertible over the integers. The Gram matrix with respect to the base change $A$ is $$A^tGA \ = \ \frac{2n+2}{d^2} \begin{pmatrix} 1 & 0 \\ 0 & 0 \end{pmatrix} \, . $$
			Therefore we have an isomorphism $L_{n,d} \cong H(\alpha, \iota)$ of lattices via $(x,y)^t \mapsto A (x,y)^t \cdot (v,u)$ where the product $\cdot$ is seen as a formal euclidean product. In particular $(d,b)$ is mapped to $v$ i.e. $\vartheta(\alpha) = [(L_{n,d}, (d,b))]$.
			
			Now let $b'$ be any integer satisfying the assumptions in (v). We know that $(d,b')$ is primitive and that
			$$ u' \ \coloneqq \ \frac{1}{d}( b'v - \iota(\alpha)) $$
			is integral, therefore $ u' - u = \frac{b' - b}{d} v$ is also integral. Since $v$ is primitive, $d$ must divide $b' - b$. Set $c \coloneqq \frac{b' - b}{d} \in \IZ$. Then
			$$ g_c \coloneqq  \begin{pmatrix} 1 & 0 \\ c & 1 \end{pmatrix}  \in \OO(L_{n,d}) $$
			is clearly an isometry of $L_{n,d}$ with $g_c(d,b)^t =(d, dc +b)^t = (d, b')^t$. Hence $\vartheta(\alpha) = [L_{n,d}, (d,b)] = [L_{n,d}, (d,b')]$. 
			
		\end{enumerate} \end{proof}

		\begin{lem}\label{lnd} The degenerate lattice $L_{n,d}$ embeds primitively and isometrically into $\mukai = \UU^{\oplus 4}$ uniquely up to an isometry in $\OO(\mukai)$. \end{lem}
		\begin{proof} Follows by \cite[Thm.\,2.9]{yellowbook}. \end{proof}

		\begin{lem}\label{root} Let $\alpha \in \Lambda = \UU^{\oplus 3} \oplus \left\langle -2n-2 \right\rangle$ be a primitive isotropic element in the Kummer type lattice. Then there exists a $u \in \Lambda$ such that $(u,\alpha) = 0$ and $(u,u) = \pm 2$. \end{lem}
		\begin{proof} Write $\alpha = \alpha_0 + \alpha_1$ with $\alpha_0 \in \UU^{\oplus 3}$ and $\alpha \in \left\langle -2n-2\right\rangle$. The discriminant of $\UU^{\oplus 3}$ is trivial since its unimodular, hence by Eichler's criterion \cite[10.]{eichler} the $\OO(\UU^{\oplus 3})$--orbit of $\alpha_0$ is determined by it's length $(\alpha_0, \alpha_0) = 2n+2$. So there exits an isometry $g \in \OO(\UU^{\oplus 3})$ such that $g(\alpha_0) = ((1,n+1), 0, 0) \in \UU^{\oplus 3}$. Set $u := g^{-1}(0,0, (1, \pm 1)) \in \UU^{\oplus 3} \subset \Lambda$. Then $(u, \alpha) = (u, \alpha_0) = 0$ and $(u,u) = ((1, \pm 1), (1, \pm 1)) = \pm 2$. \end{proof}

		For a positive integer $d$ let $I_d(X) \subset \Lambda = H^2(X,\IZ)$ denote the subset of primitive isotropic elements $\alpha$ such that $\Div(\alpha) = d$ which is clearly a $\Mon^2(X)$--invariant subset. Let $\Sigma_{n,d}$ denote the set of isometry classes of pairs $(H,w)$ such that $H$ is isometric to $L_{n,d}$ and $w \in H$ is a primitive class with $(w,w) = 2n+2$. 
		\begin{theo}\label{moninv} Let $X$ be a Kummer type manifold of dimension $2n$ and $d$ a positive integer such that $d^2$ divides $n+1$. The map 
			$$\vartheta \ : \ I_d(X) \longrightarrow \Sigma_{n,d} \, , \ \ \ \alpha \longmapsto \vartheta(\alpha) \ = \ \left[ \left( H(\alpha, \iota), v \right)  \right] $$
			is a surjective faithful monodromy invariant of the manifold $X$. 
		\end{theo} 
		
		\begin{proof} By Lemma \ref{welldefined} $\vartheta : I_d(X) \longrightarrow \Sigma_{n,d}$ is well defined and $\Mon^2(X)$--invariant. 
			
			To show that $\vartheta$ is faithful i.e. that the induced map $\vartheta \ : \ I_d(X)/\Mon^2(X) \longrightarrow \Sigma_{n,d}$ is injective, we assume $\alpha_1, \alpha_2 \in I_d(X)$ with $\vartheta(\alpha_1) = \vartheta(\alpha_2)$, that means we have an isometry $g : H(\alpha_1, \iota) \rightarrow H(\alpha_2, \iota)$ with $g(v) = v$ where $v$ is as usual a generator of $\iota(\Lambda)^{\bot}$. 
			
			We first show that both $\alpha_i$ lie in the same $\kw(\Lambda)$--orbit, where the group $\kw(\Lambda)$ was defined in Definition \ref{wx}. We have $H(\alpha_i, \iota) \cong L_{n,d}$. By Lemma \ref{lnd} there is up to an isometry in $\OO(\mukai)$ a unique way to embed $H(\alpha_i, \iota)$ isometrically and primitively into $\mukai$, hence we can extend $g$ to an isometry $\tilde{g} \in \OO(\mukai)$. Since $v^\bot = \iota(\Lambda)$ we have in particular $\tilde{g}(\iota(\Lambda)) = \iota(\Lambda)$, i.e it makes sense to set $h := \iota^{-1} \circ \tilde{g} \circ \iota$ which is an isometry $h \in \OO(\Lambda)$ such that $\iota \circ h = \tilde{g} \circ \iota$, hence $h$ leaves the orbit $\iota_X = \OO(\mukai) \iota$ invariant and by Lemma \ref{stabil} either $\mu = h$ or $\mu = -h$ is contained in the subgroup $\kw(\Lambda)$ of orientation preserving isometries acting as $\pm 1$ on the discriminant $\Lambda^{\vee}/\Lambda$. Choose $\mu$ such that it is in $\kw(\Lambda)$. The null space of $H(\alpha_i,\iota) \subset \mukai$ is generated by $\iota(\alpha_i)$. Since $\tilde{g} \in \OO(\mukai)$ restricts to an isometry between $H(\alpha_1,\iota)$ and $H(\alpha_2,\iota)$ the null space of $H(\alpha_2,\iota)$ is generated by $\tilde{g}(\iota(\alpha_1))  = \iota(\pm h(\alpha_1)) = \iota(\mu(\alpha_1))$. So we have $\iota(\mu(\alpha_1)) = \pm \iota(\alpha_2)$, hence $\mu(\alpha_1) = \pm \alpha_2$. By Lemma \ref{root} we can choose a $u \in \Lambda$ with $(u,\alpha_2) = 0$ and $(u,u) = +2$. Then the isometry $\rho_u \in \OO(\Lambda)$ defined in Definition \ref{orientreflection} i.e. $\rho_u(x) = -R_u(x) = -x+(u,x)u$ is contained in $\kw(\Lambda)$, see Corollary \ref{nevercontained} and Remark \ref{determinant}, and satisfies $\rho_u(\alpha_2) = -\alpha_2$, hence $$\kw(\Lambda) \alpha_1 \ = \ \kw(\Lambda) (\pm \alpha_2) \ = \ \kw(\Lambda) \alpha_2 \, . $$ 
			Now we show that $\kw(\Lambda) \alpha = \Mon^2(X) \alpha$ for every primitive isotropic element $\alpha \in \Lambda$. Since $\Mon^2(X) \subset \kw(\Lambda)$ is an index $2$ subgroup by Corollary \ref{nevercontained} we can write
			$$ \kw(\Lambda) \ = \ \Mon^2(X) \ \cup \  \Mon^2(X) w $$
			for every $w \in \kw(\Lambda) \setminus \Mon^2(X)$. By Lemma \ref{root} we have an element $u \in \Lambda$ with $(u,\alpha) = 0$ and $(u,u) = -2$. Then the reflection $\rho_u(x) = R_u(x) = x + (u,x)u$ of $\Lambda$ (Definition \ref{orientreflection}) acts as $+1$ on the discriminant but has determinant $-1$, hence it is contained in $\kw(\Lambda)$ but not in $\Mon^2(X)$, see again Corollary \ref{nevercontained} and Remark \ref{determinant}. In particular $\rho_u(\alpha) = \alpha$ therefore
			\begin{align*} \kw(\Lambda) \alpha \ & = \ \left(\Mon^2(X) \ \cup \ \Mon^2(X) \rho_u \right) \alpha \\
			\ & = \ \Mon^2(X) \alpha \ \cup \  \Mon^2(X) \rho_u(\alpha) \ = \ \Mon^2(X) \alpha \, . \end{align*}
			
			For surjectivity, assume we have a class $[(L_{n,d}, w)] \in \Sigma_{n,d}$, i.e. $w \in L_{n,d}$ is primitive such that $(w,w) = 2n+2$. By Lemma \ref{lnd} there exists a primitive isometric embedding $\iota_{n,d} : L_{n,d} \hookrightarrow \mukai$. 
			
			By Eichler's criterion \cite[10.]{eichler} we can assume that $\iota_{n,d}(w)$ is contained in a copy of $\UU$ of $\mukai = \UU^{\oplus 4}$. Then the lattice $\iota_{n,d}(w)^{\bot} \subset \mukai = \UU^{\oplus 4}$ is of signature $(3,4)$ and since $(w,w) = 2n+2$ the complement $\iota_{n,d}(w)^{\bot}$ is isomorphic to $\Lambda \cong \UU^{\oplus 3} \oplus \left\langle -2n-2\right\rangle$. 
			
			The action of $\OO(\Lambda) \times \OO(\mukai)$ on $\OO(\Lambda, \mukai)$ is transitive by Remark \ref{uniquetransitive}, hence the induced action of $\OO(\Lambda)$ on the orbit set $\OO(\Lambda, \mukai)/\OO(\mukai)$ is also transitive. Hence, we can choose an isometry $g : \iota_{n,d}(w)^{\bot} \rightarrow \Lambda$ such that $$\kappa \ : \ \Lambda \stackrel{g^{-1}}{\longrightarrow} \iota_{n,d}(w)^{\bot} \subset \mukai$$ belongs to the monodromy invariant orbit $\iota_X = \OO(\mukai) \iota$. Recall from above that $(0,1) \in \ker L_{n,d}$ is the generator of $\ker L_{n,d}$. Clearly we have $(w,(0,1)) = 0$ in $L_{n,d}$ so we can set $\alpha := g(\iota_{n,d}(0,1))$. We can write
			$$ \alpha \ = \ a \xi + b \delta $$
			where $\xi \in \UU^{\oplus 3}$, $\delta \in \left\langle -2n-2\right\rangle$, $a > 0$ such that $\gcd(a,b) = 1$. As in the proof of Lemma \ref{welldefined} (iv) it follows that $a = \Div(\alpha)$. We have $\kappa(\Lambda)^{\bot} = \left\langle \iota_{n,d}(w)\right\rangle$ and $\kappa(\alpha) = \iota_{n,d}((0,1))$ and from Lemma \ref{welldefined} again 
			$$H(\alpha,\kappa) \ = \  \sat\left\langle \iota_{n,d}(0,1), \iota_{n,d}(w) \right\rangle \ \cong \ L_{n,a} \, ,$$
			where $\iota_{n,d}(w)$ is mapped to $(a,b)$. The primitive element $w \in L_{n,d}$ is necessarily of the form $(\pm d, w_2)$ with $\gcd(d, w_2 ) = 1$. Over the rational numbers we have clearly $\iota_{n,d}(L_{n,d})_{\IQ} = \left\langle \iota_{n,d}(0,1), \iota_{n,d}(w) \right\rangle_{\IQ}$. As $\iota_{n,d}(L_{n,d})$ is saturated it follows that
			$$\iota_{n,d}(L_{n,d}) \ = \ \sat\left\langle \iota_{n,d}(0,1), \iota_{n,d}(w) \right\rangle \ = \ H(\alpha, \kappa) \, .$$
			Now we have an isometry
			$$ L_{n,d} \stackrel{\iota_{n,d}}{\hookrightarrow} H(\alpha, \kappa) \rightarrow L_{n,a} $$
			where $w$ is mapped to $(a,b)$, hence $\Div(\alpha) = a = d$ i.e. $\alpha \in I_d(X)$ and $\vartheta(\alpha) = [(L_{n,d}, w)]$. \end{proof}
	\end{ueber} 

	\section{Beauville--Mukai systems of generalized Kummer type}\label{bminmoduli}\label{bmgenkum}
	
	We define the notion of a \emph{Beauville--Mukai system of generalized Kummer type}. It is similarly defined as in the $\kdrei$ case, see \cite{benni1}. The fibers of them are not Jacobian of curves anymore, but an abelian subvariety of a Jacobian. Therefore we dwell on some theory of complementary subvarieties in Jacobians, see subsection \ref{jactheo}. \\
	
	 Let $S$ be an abelian surface and $v$ be a primitive Mukai vector on $S$ of the form $v = (0, c_1(D), s)$ where $D$ is an ample divisor on $S$ \ie $D$ is ample. We set $2n \coloneqq (D,D) - 2$. Note that we have $h^0(S,D) = \frac{1}{2} (D,D) = n + 1$. Choose a $v$--generic ample class $H$ on $S$, hence $M_H(v)$ is an holomorphic symplectic manifold as explained in section \ref{anorbit}. 
		
		For simplicity we now fix a reference point $F_0 \in M_H(v)$ such that $\det(F_0) = \ko_S(D)$. By \cite{yoshi} the Albanese map $\alb_v : M_H(v) \rightarrow S \times S^\vee$ with respect the reference point $F_0 \in M_H(v)$ can be written as 
		$$ (\alb_v)_{F_0} \ = \ \alpha \times \det\nolimits_{F_0} $$
		where $\det_{F_0} : M_H(v) \rightarrow \Pic^0(S) = S^\vee$ is defined as $\det_{F_0}(F) \coloneqq \det(F) \otimes (\det(F_0))^{-1}$ and $\alpha$ can be defined as
		\begin{equation}\label{albha} \alpha(F) \ \coloneqq \ \sum c_2(F) \ \coloneqq \ \sum_i n_i x_i \end{equation}
		where we view $c_2(F)$ in the Chow ring represented by the cycle $[\sum_i n_i x_i]$, see \cite[4.1 ff.]{yoshi} and \cite[p. 11]{ogradygael2}. 
		
		The Albanese fiber $K_H(v) = (\alb_v)^{-1}_{F_0}(0,0)$ is an irreducible holomorphic symplectic manifold of dimension $2n$ see section \ref{anorbit} and for $F \in K_H(v)$ the fitting support $\supp(F)$ is an element of the linear system $|D|$. 
		This leads to the following commutative diagram 
		\begin{equation}\label{bmdiagram}   \xymatrix{  K_H(v) \ar@{^{(}->}[r] \ar[d]_{\pi} & M_H(v) \ar[d]_{\pi} \ar[r]^{(\alb_v)_{F_0}} & S \times S^\vee \ar[d]^{\pr_{S^\vee}} \\ 
			|D| \ar@{^{(}->}[r]	&	\{D\} \ar[r] & S^\vee = \Pic^0(S) 		   } \end{equation}
		where $\{D\} \rightarrow S^\vee$ is the map $C \mapsto \ko_S(C) \otimes \det(F_0)^{-1}$. The induced map $K_H(v) \rightarrow |D|$ is a Lagrangian fibration by Matsushita's Theorem \ref{mats}.
		\begin{defin} In the setting as above, the Lagrangian fibration $$\pi : K_H(v) \longrightarrow |D| \, , \ \ \  F \longmapsto \supp(F)$$ is called a \emph{Beauville--Mukai system of generalized Kummer type}. 	\end{defin}

		\begin{ueber}\textbf{An excursion to the theory of Jacobians.}\label{jactheo} To consider the fibers of Beauville--Mukai systems of generalized Kummer type and polarizations on them, we deal with some theory of complementary abelian subvarieties.

		If $M$ is an abelian variety, $A \subset M$ is an abelian subvariety and $L$ a polarization on $M$, then one can define a so called \emph{complementary subvariety} $B$ to $A$ (with respect to $L$). We only consider the case when $L = \Theta$ is a principal polarization \cite[12.1]{BL}, for the more general setting see \cite[5.3]{BL}. We denote the induced isogeny of $L$ by $\phi_L$.
		
		
		We assume for this section, that $\Theta$ is a principal polarization, therefore we can identify $M$ with its dual $M^{\vee}$ via the homomorphism $\phi_{\Theta}$. By \cite[Prop. 1.2.6]{BL} for any polarization $L$ the isogeny $\phi_L$ has always a $\IQ$--inverse and we can define the $\IQ$--endomorphism $$g_A \coloneqq \iota \circ \phi^{-1}_{\iota^{\star} L} \circ \iota^{\vee} : M \otimes \IQ \longrightarrow M \otimes \IQ$$ where $\iota = \iota_A : A \hookrightarrow M$ denotes the inclusion. Choose a positive number $m$ such that $m g_A$ is an endomorphism of $M$. By \cite[Prop. 12.1.3]{BL} we have 
		\begin{equation}\label{complement} B \ \coloneqq  \ (\ker( m g_A))_0 \subset M \ = \ \ker \iota^{\vee} \ \cong \ (A/B)^{\vee}  \, ,\end{equation}
		where $(\ker (m g_A))_0$ denotes the identity component. Furthermore, $B$ is an abelian subvariety of $M$ called the \emph{complementary subvariety} to $A$ (with respect to $L$). Conversely, $A$ is also the complementary subvariety to $B$ and $(A,B)$ is called a \emph{pair of complementary subvarieties} in $M$.
		
		\begin{prop}\cite[Cor.~12.1.5]{BL}\label{indupol} Let $(A,B)$ be a pair of complementary abelian subvarieties in a principally polarized abelian variety $(M,\Theta)$ with $\dim A \geq \dim B = r$. Denote by $\iota_A$ and $\iota_B$ the inclusions of $A$ and $B$ into $M$ respectively and assume $\pol(\iota_B^{\star} \Theta) = (d_1, \ldots, d_r)$. Then $\pol(\iota_A^{\star} \Theta) = (1, \ldots, 1, d_1, \ldots, d_r)$. \end{prop}
		
		\begin{ueber}\textbf{The case of a Jacobian.}\label{casejacobian} Let $\iota : C \hookrightarrow S$ be a smooth curve in an abelian surface $S$. Denote by $\Theta$ the principal polarization of the Jacobian $\Jac(C)$ of $C$ and define $K(C) \coloneqq \ker \Jac(\iota)$ to be the kernel of the homomorphism $\Jac(\iota)$ induced by the inclusion $\iota : C \hookrightarrow S$ and by the universal property of the Jacobian \cite[11.4.1.]{BL}, \ie we have an exact sequence
			$$ K(C) \hookrightarrow \Jac(C) \stackrel{\Jac(\iota)}{\longrightarrow} S \, .$$
			Using the several identifications of the dual and double dual, the dual of the pullback or the double pullback
			$$ (\iota^{\star})^{\vee} = (\iota^{\star})^\star \ : \ \Jac(C) = (\Jac(C))^{\vee} = \Pic^0(C) \longrightarrow S = (S^{\vee})^{\vee} = \Pic^0(S^\vee) $$
			is nothing but the map $\Jac(\iota)$. Of course, you can also see $\Jac(\iota)$ as the Albanese map induced by $\iota$
			$$\alb(\iota)  \ : \ \alb(C) \longrightarrow \alb(S) = S $$
			if you identify $\alb(C)$ and $\Jac(C)$, cf. \cite[Prop.~11.11.6]{BL}.
			More concretely, the map $\Jac(\iota)$ viewed as a map $\Pic^0(C) \rightarrow S$ is
			\begin{equation}\label{jaci} \ko_C\left(  \sum n_i x_i\right)  \longmapsto \sum n_i x_i \, . \end{equation}
			Indeed, for a given point $c$, denote by $\alpha_c : C \hookrightarrow \Jac(C)$, $x \mapsto \ko_C(x-c)$ the Abel--Jacobi map. Then
			$$ \Jac(\iota) ( \alpha_c(x) )  \ = \ \Jac(\iota)( \ko_C(x - c)) \ = \ x - c \ = \ t_{-c}(x) \, ,$$
			therefore it satisfies exactly the property of the unique morphism as described in \cite[11.4.1.]{BL}.
			
			We can see the dual $S^\vee$ as an abelian subvariety of $\Jac(C)$ in the following sense. 
			
			\begin{lem}\label{pullinjective} The pullback morphism $\iota^{\star} : S^{\vee} \rightarrow \Jac(C)$ is an injection. Therefore $K(C)$ is connected.  \end{lem}
			\begin{proof} Let $L$ be a line bundle on $S$ with $L|_C = \ko_C$. We have the standard exact sequence	
				\begin{equation} 0 \longrightarrow L \otimes \ko_S(-C) \longrightarrow L \longrightarrow L|_C = \ko_C \longrightarrow 0 \, . \label{stdex}\end{equation}		
				Since $C$ is effective, the line bundle $L(-C) = L \otimes \ko_S(-C)$ has no holomorphic sections \ie $H^0(S, L(-C)) = 0$. In particular, $L(-C)$ cannot be ample (cf. \cite[Prop.~4.5.2]{BL}), therefore the associated hermitian form of $c_1(L(-C))$ must have less then four positive eigenvalues. By \cite[Lem.~3.5.1]{BL} we then have $H^1(S, L^{\vee}(-C)) = 0$. The long exact sequence of \refb{stdex} shows that $h^0(S, L) = h^0(S, \ko_C) = 1$ \ie $L$ has a holomorphic section $s$. Since $0 = c_1(L|_C) = [V(s)]$, the zero set $V(s)$ of $s$ is empty \ie $L = \ko_S(V(s)) = \ko_S(0) = \ko_S$.	
				
				For the second statement identify $\Jac(C)^\vee = \Jac(C)$ via the principal polarization. We have the short exact sequence 
				$$ 0 \longrightarrow S^{\vee} \longrightarrow \Jac(C) \longrightarrow \Jac(C)/S^{\vee} \longrightarrow 0 $$
				and by \cite[Prop.~2.4.2]{BL} the dual sequence
				$$ 0 \longrightarrow (\Jac(C)/S^{\vee})^\vee \longrightarrow \Jac(C) \longrightarrow S \rightarrow 0 $$
				is also exact. Hence $K(C) = \ker( \Jac(C) \rightarrow S) \cong (\Jac(C)/S^{\vee})^\vee$ \ie $K(C)$ is connected. \end{proof}
			
			In other words we have the following.
			
			\begin{lem}\label{comp} The abelian subvarieties $K(C)$ and $S^{\vee} \stackrel{\iota^{\star}}{\hookrightarrow} \Jac(C)$ are a pair complementary abelian subvarieties of $\Jac(C)$. 
			\end{lem}
			\begin{proof} With the discussion above we have $K(C) = \ker (\iota^{\star})^{\vee}$ which is exactly the definition as in \refb{complement} \end{proof}
			We are interested in the type of the polarization induced by $\Theta$.
			
			\begin{lem}\label{typeofc}  Let $L$ be a polarization on an abelian surface $S$ of type $\pol(L) = (d_1, d_2)$. Then $h^0(S,L) = d_1 d_2$. If $C \in |L|$ is a not necessarily smooth curve, then we have for its arithmetic genus $g_a = d_1d_2 + 1$. Furthermore, if $c_1(L)$ is primitive and $(L,L) = 2d$, then $\pol(L) = (1,d)$.
				
			\end{lem}
			\begin{proof}  By the well known formula for the (arithmetic) genus, we have $g_a = 1 + \frac{1}{2}(C,C)$.  By the geometric Riemann--Roch \cite[3.6 ff.]{BL} and since $L$ is ample, we have
				$$ d_1 d_2 \ = \ \chi(L) \ = \ h^0(S,L) \ = \ \frac{1}{2} (C,C) \ = \ g_a - 1 \, .$$
				If $(L,L) = 2d$ and $c_1(L)$ is primitive, the equation above also shows $2d = (L,L) = 2d_1 d_2$. Since $d_1$ divides $d_2$ and $(d_1, d_2)$ is primitive as $c_1(L)$ is primitive, we have $d_1 = 1$ \ie $\pol(L) = (1,d)$.   
				\end{proof}

			
			

			\begin{rem}\label{dualpol} If $L$ is a polarization on an abelian variety $S$ with $\pol(L) = (d_1,\ldots,d_n)$, then by \cite[14.4]{BL} there is a natural polarization $L_{\delta}$ on the dual $S^\vee$, called \emph{dual polarization}, characterized by the following equivalent properties
					$$ \text{(a)} \ \ \phi^\star_L L_{\delta} \text{ is algebraically equivalent to } L^{d_1 d_n} \, , \ \ \ \text{(b)} \ \ \phi_{L_{\delta}} \phi_L = d_1 d_n \id_S \, .$$
					Furthermore, the type is given by $\pol(L_{\delta}) = (d_1, \frac{d_1 d_n}{d_{n-1}}, \ldots, \frac{d_1 d_n}{d_2}, d_n)$. If we are on an abelian surface, then obviously $\pol(L) = \pol(L_\delta)$.
			\end{rem}

			\begin{lem}\label{multiple} Let $(S,L)$ denote a polarized abelian surface of type $\pol(L) = (d_1,d_2)$. Then for every smooth curve $\iota : C \hookrightarrow S$ with $C \in |L|$ the restriction $$\Theta|_{S^\vee} \ \coloneqq \ (\iota^{\star})^\star \Theta $$ is a polarization of type $(d_1,d_2)$, where $\Theta$ denotes the principal polarization on $\Jac(C) = \Pic^0(C)$ and $(\iota^\star)^\star$ is viewed as a map $\Pic(\Jac(C)) \rightarrow \Pic(S^\vee)$. In particular, if the Picard number is $\rho(S) = 1$, then $\Theta|_{S^\vee} = L_{\delta}$ where the latter is the dual polarization on $S^\vee$ to $L$, cf. \emph{Remark \ref{dualpol}}.
			\end{lem}	
			\begin{proof} The proof is divided in three steps. In the first, we assume for the Picard number $\rho(S) = 1$ and show the existence of such a curve in $|L|$. In the second and still under the assumption $\rho(S) = 1$ it is shown that it holds for every smooth curve in $|L|$. In the third step we drop the restriction on the Picard number.
				We set $d \coloneqq d_1d_2$. 
				
				\begin{itemize} \item We first assume $\rho(S) = 1$ for the Picard number and prove the existence of such a curve $C \in |L|$. Since $\rho(S) = 1$ we have also $\rho(S^\vee) = 1$. 	Note that we have $\pol(L_{\delta}) = (d_1, d_2)$ by Remark \ref{dualpol}. 
					
					Consider the isogeny $\phi_L : S \rightarrow S^\vee$. Then \begin{equation}
					\label{kerphi}
					\ker \phi_L \ \cong \ (\IZ/d_1\IZ \times \IZ/d_1\IZ) \oplus (\IZ/d_2 \IZ \oplus \IZ/d_2 \IZ)\end{equation} by \cite[Lem.~3.1.4]{BL}\footnote{In \cite{BL} they use the notation $K(L)$ for $\ker \phi_L$.}. On $\ker \phi_L$ we have the alternating Weil pairing $$e : \ker \phi_L \times \ker \phi_L \longrightarrow \CC^\star$$ see \cite[p.~160]{BL}, for the special case of an abelian surface see also \cite[Ex.~6.7.3]{BL}. For $[x] = ([x_i]), [y] = ([y_i]) \in \ker \phi_L$ with respect to the isomorphism in \refb{kerphi}, the pairing $e$ can be calculated as
					$$ e([x], [y]) \ = \ \exp\left(\frac{2\pi \ima}{d_1} (x_3 y_1 - x_1 y_3)  \right) \cdot \exp\left( \frac{2\pi \ima}{d_2} (x_4 y_2 - x_2 y_4 )  \right) \, ,$$
					see	\cite[Ex.~6.7.3]{BL}.
					
					Choose a subgroup $G \subset \ker \phi_L$ such that $G \cong \IZ / d_1\IZ \oplus \IZ / d_2 \IZ$ and which is isotropic with respect to the pairing above.

					Then $\phi_L$ factorizes as
					$$   \begin{xy}\xymatrix{  S \ar[rr]^{\phi_L} \ar[dr]_{p} & & S^\vee  \\ 
						&	S/G \ar[ur]_{p^\star} & 		}   \end{xy}$$
					where $p$ is the canonical projection which is a $d = d_1 d_2$ to $1$ map. As $G$ is isotropic, by \cite[Prop.~6.7.1]{BL} the action of $G$ on $S$ lifts to a free action of $G$ on $L$, in particular we can define $L_0 \coloneqq L/G \in \Pic(S/G)$. Since $p$ is of degree $d$ and $p^\star L_0 = L$, we have 
					$$ 2d \ = \ (L,L) \ = \ (p^\star L, p^\star L) \ = \ d(L_0, L_0) \, , $$
					\ie $(L_0, L_0) = 2$, therefore $L_0$ is a principal polarization on $S/G$. Hence, $H^0(S/G,L_0) = \CC \sigma$ for a nontrivial section $\sigma$. Define the curve $C_0 \coloneqq V(\sigma)$. Then $C_0$ is an element of $|L_0|$ and we claim that $C_0$ is smooth and irreducible. 
					
					Indeed, assume $C_0 = C_1 + C_2$. Since $\rho(S/G) = 1$ we have $C_1 = m_1 L_0$ and $C_2 = m_2 L_0$ with positive integers $m_i$. Then 
					$$ 2 \ = \ C_0^2 \ = \ (C_1 + C_2)^2 \ = \ ({m_1} + m_2)^2 (L_0, L_0) \ = \ 2({m_1}^2 + m_2^2 + 2m_1 m_2) > 2 \, , $$
					which is absurd. 
					
					If $C_0$ is not smooth, then let $\nu : \widetilde{C}_0 \rightarrow C_0$ be its normalization. For its  genus we have $g(\widetilde{C}_0) < g_a(C_0) = 2$. If $g(\widetilde{C}_0) = 0$, then $\widetilde{C}_0 = \IP^1$ which is absurd, since $\nu : \IP^1 \rightarrow C_0 \hookrightarrow S/G$ would be a non constant regular map which is not possible. If $g(\widetilde{C}_0) = 1$ then $\widetilde{C}_0$ would be an elliptic curve which can be seen as an abelian subvariety of $S/G$ after a translation, if necessary. Then $\widetilde{C}_0$ has a complementary abelian subvariety in the sense as above. This would mean $\rho(S/G) \geq 2$ which contradicts $\rho(S/G) = 1$. 
					
					We conclude that $C_0$ is irreducible and smooth. In particular, $C_0$ is of genus $2$ and by Lemma \ref{pullinjective}, $(S/G) \cong (S/G)^\vee$ embeds into $\Jac(C_0)$. Both have the same dimension, hence $S/G \cong \Jac(C_0)$.

					Set $C \coloneqq p^{-1}(C_0)$. Then $C$ is an element of $|L|$ as $L = p^\star L_0$ and is smooth as $p$ is etale. It has to be connected with a similar argument as above. Assume $C = C_1 \cup C_2$ is a disjoint union. As $\rho(S) = 1$ we have $C_i = m_i L'$ for positive integers $m_i$ where $L'$ is the primitive part of $L$. Then
					$$ 0 \ = \ (C_1, C_2) \ = \ (m_1 L', m_2 L') \ = \ 2 m_1 m_2 \frac{d_2}{d_1} > 0 \, , $$ 
					which is absurd. 
					
					Hence, $C$ is a connected smooth curve.

					Denote by $\iota : C \hookrightarrow S$ the inclusion, by $q \coloneqq p|_{C} = p \circ \iota : C \rightarrow C_0$ the induced $d$ to $1$ cover and by $\Theta_0$ the principal polarization on $\Jac(C_0)$. Since $\rho(S) = 1$, also $\rho(S^\vee) = 1$ and $\rho(\Jac(C_0)) = 1$, so we have for the pullback $(p^\star)^\star L_{\delta} = k \Theta_0$ for some positive integer $k$. As $p^\star$ is surjective of degree $d$, taking the self intersection on both sides gives
					$$2 k^2 \ = \ (k\Theta_0, k\Theta_0) \ = \ ((p^\star)^\star L_{\delta}, (p^\star)^\star L_{\delta} ) \ = \ d (L_{\delta}, L_{\delta}) \ = \ 2d^2$$
					and hence $k = d$ \ie \begin{equation}\label{beh1} (p^\star)^\star L_{\delta} \ = \ d \Theta_0 \, . \end{equation} As $q = p \circ \iota$ we can split $q^\star$ as
					$$ q^\star : \Jac(C_0) \stackrel{p^\star}{\longrightarrow} S^\vee \stackrel{\iota^\star}{\longrightarrow} \Jac(C) \, . $$
					Since $\rho(\Jac(C_0)) = 1$, we have $(q^\star)^\star \Theta = a \Theta_0$ for some positive integer $a$. By \cite[Lem.~12.3.1]{BL} $(q^\star)^\star \Theta$ is algebraically equivalent to $d \Theta_0$. Therefore $ac_1(\Theta_0) = c_1((q^\star)^\star \Theta) = d c_1(\Theta_0)$, hence $a = d$ \ie \begin{equation}
					\label{beh2}(q^\star)^\star \Theta \ = \ d \Theta_0 \, .\end{equation} Finally write $(\iota^\star)^\star \Theta = b L_{\delta}$ for some positive integer $b$. We have
					$$d \Theta_0 \ \stackrel{\emph{\refb{beh2}}}{=} \ (q^\star)^\star \Theta \ = \ (\iota^\star \circ p^\star)^\star \Theta \ = \ (p^\star)^\star (\iota^\star)^\star \Theta \ = \ (p^\star)^\star(b L_{\delta}) \ \stackrel{\emph{\refb{beh1}}}{=}  \ bd \Theta_0 \, , $$
					hence $b = 1$ \ie  $(\iota^\star)^\star \Theta = L_{\delta}$.
					\item We show that the statement holds for every element in $|L|$ but still assume $\rho(S) = 1$ for the Picard number. 
					
					Consider the open and connected set $U \subset |L| \cong \IP^{d-1}$ such that every element in $U$ corresponds to a smooth curve in $S$. Let $\kc \rightarrow U$ be the associated family of smooth curves. We can take the relative Jacobian $$\pi_k : X^k \coloneqq \Pic^{k}(\kc/U) \longrightarrow U$$ of degree $k \in \IZ$ of it. 
					
					By Lemma \ref{typeofc} the genus of $\kc_t$ is $g = d + 1$. By considering the image of $(X^d)^{(d)} \rightarrow X^d$, $(x_1,\ldots,x_{d}) \mapsto \sum_i x_i$ which is a divisor in $X^d$, we obtain a line bundle $\km \in \Pic(X^d)$ such that $\km_t \coloneqq \km|_{X^d_t}$ is the natural polarization on $X^d_t = \Pic^{d}(\kc_t)$. 
					
					Locally we can identify $X^d$ with $X^0$, say $X^d_V = \pi_d^{-1}(V) \cong X^0_V = \pi_0^{-1}(V)$ where $V \subset U \subset |L|$ is chosen connected, by twisting with a line bundle $Q^V$ on $\pi_d^{-1}(V)$ which has degree $-d$ on the fibers $\kc_t$ for $t \in V$. Then we obtain on a line bundle $\kl^V = \km \otimes Q^V$ on $X^0_V$, such that $\kl^V_t \coloneqq \kl^V|_{X^0_t}$ is the principal polarization on $X^0_t = \Jac(\kc_t)$ for $t \in V$. 
					Let $\iota_t : \kc_t \hookrightarrow S$ denote the inclusion. Then the self intersection $m_V : V \rightarrow \IZ$
					\begin{equation}
					\label{selfintersection} 
					m_V(t)  \ \coloneqq \ \left((\iota_t^\star)^\star \kl^V_t , (\iota_t^\star)^\star \kl^V_t\right) \end{equation}
					of $(\iota_t^\star)^\star \kl^V_t$ is a continuous and integer valued function, therefore must be constant as $V$ is chosen connected.

					By the first part we know that there is an element $t_0 \in U$ such that the statement for the curve $\kc_{t_0}$ holds. For arbitrary $t_N \in U$, choose a path $\gamma$ from $t_0$ to $t_N$ in $U$. By the discussion above, we can cover $\gamma$ with finitely many connected open sets $V_0, \ldots, V_N$ such that $t_0 \in V_0$ and $t_N \in V_N$ and we have elements $t_i \in V_i \cap V_{i+1}$ for $i = 1, \ldots N-1$. Then the self intersections $m_{V_i}$ and $m_{V_{i+1}}$ must coincide on $V_i \cap V_{i+1}$. 
					
					By assumption, we have
					$$m_{V_0}(t_0) \ = \ (L_{\delta}, L_{\delta}) \ = \ 2d$$ \ie $m_{V_0} \equiv 2d$. Assume $(\iota_{t_N}^\star)^\star \kl_{t_N} = k L_{\delta}$ for some positive integer $k$. Then
					$$ 2d = m_{V_0}(t_0) \ = m_{V_1}(t_1) = \ldots = m_{V_N}(t_N) = k^2 2d$$
					\ie $k = 1$.

					\item We now consider the general case \ie let $S$ be with arbitrary Picard number. We have an universal family $p : \kx \rightarrow \gh_2$ of $(d_1,d_2)$--polarized abelian surfaces over Siegel's upper half plane $\gh_2$, see \cite[8.7]{BL}. Let $\kn$ denote the line bundle on $\kx$ such that $\kn_s \coloneqq \kn|_{\kx_s}$ is the $(d_1,d_2)$ polarization on $\kx_s$ for $s \in \gh_2$.

					For each $s \in \gh_2$ let $U_s \subset |\kn_s| \cong \IP^{d-1}$ be the open set such that all elements in $U_s$ corresponds to smooth curves in $\kx_s$. Let $U \subset \IP^{d-1}$ denote the open and connected subset such that for every for every $(s, t) \in \gh_2 \times U$ the point $t \in \IP^{d-1} \cong |\kn_s|$ corresponds to an element in $U_s$. In particular it corresponds to a smooth curve $\kc^s_{t}$ in $\kx_s$. Let $\iota_{s,t} : \kc^s_{t} \hookrightarrow \kx_s$ denote the inclusion.
					
					From the second step of the proof we know that for each $(s,t) \in \gh_2 \times \IP^{d-1}$ we can find a neighbourhood $V_{s,t} \subset U_s$ of $t$ and a relative principal polarization $\kl^{{s,t}}$ on $\Pic^0(\kc^s/|U_s)|_{V_{s,t}}$ where $\kc^s \rightarrow U_s$ denotes the associated family of smooth curves to $U_s$. 
					
					We can define the map
					$$ \varphi : \gh_2 \times U \longrightarrow \IZ^2 \, , \ \ \ (s,t) \longmapsto \pol\left( (\iota_{s,t}^\star)^\star \kl^{s,t}_t \right) $$
					for the case that $(s,t) \in \gh_2 \times V_{s,t}$. This is well defined and continuous, therefore must be constant as $U$ is connected. It is well known, see \cite[8.11, (1)]{BL}, that the generic abelian surface has endomorphism ring $\End = \IZ$ \ie has Picard number $\rho = 1$, by Lemma \ref{picard}. Therefore the statement proven in the second step applies for a generic element $(s_0, t_0) \in \gh_2 \times U$ \ie $\varphi(s_0, t_0) = (d_1, d_2) \equiv \varphi$.  
					
					For our original situation this means that the type of $\Theta|_{S^\vee} = (\iota^\star)^\star \Theta$ is $\pol(\Theta|_{S^\vee}) = (d_1, d_2)$ for arbitrary $(d_1, d_2)$--polarized $(S,L)$.
					
					\end{itemize}
			\end{proof}

			An immediate consequence of Lemma \ref{multiple} and Proposition \ref{indupol} is the following.
			
			\begin{prop}\label{kc} Let $(S,L)$ denote a polarized abelian surface of type $\pol(L) = (d_1,d_2)$. Then for every smooth curve $C \in |L|$, the type of the restriction of the principal polarization $\Theta$ of $\Jac(C)$ to $K(C)$ is
				$$\pol(\Theta|_{K(C)}) \ = \ (1, \ldots, 1, d_1, d_2) \, . $$ 	
			\end{prop}
			\begin{proof} By Lemma \ref{multiple} the restriction $\Theta|_{S^\vee}$ is a polarization of type $\pol(\Theta|_{S^\vee}) = (d_1, d_2)$. By Proposition \ref{indupol}, the type of $\Theta|_{K(C)}$ is $\pol(\Theta|_{K(C)}) \ = \ (1, \ldots, 1, d_1, d_2) $. \end{proof}

		\end{ueber}

		\begin{ueber}\textbf{Fibers of Beauville--Mukai systems.} Let $\pi : K_H(v) \rightarrow |D|$ denote a Beauville--Mukai system of generalized Kummer type. Consider a smooth curve $C \in |D|$, then the fiber of the support morphism $M_H(v) \rightarrow \{D\}$ is given by the Jacobian $\Jac^d(C) $ of a certain degree $d$. The restriction of the Albanese map $(\alb_v)_{F_0} = \alpha_{F_0} \times \det_{F_0}$ to $\Jac^d(C) \subset M_H(v)$ is in the second component constant $0$. Therefore, if we denote by $K^d(C) \subset \Jac^d(C)$ the fiber of $\pi : K_H(v) \rightarrow |D|$, we have an exact sequence
		\begin{equation}\label{fiber} 0 \longrightarrow K^d(C) \hookrightarrow \Jac^d(C) \stackrel{\alpha}{\longrightarrow} S \end{equation}
		where $\alpha = \pr_S \circ (\alb_v)_{F_0}$ and the diagram
		\begin{equation}  \xymatrix{  K_H(v) \ar@{^{(}->}[r] & M_H(v) \ar[r]^{(\alb_v)_{F_0}} & S \times S^\vee  \\ 
			K^d(C) \ar@{^{(}->}[r] \ar@{^{(}->}[u]	&	\Jac^d(C) \ar[r]^\alpha \ar@{^{(}->}[u] & S \ar@{^{(}->}[u]		   }  \end{equation}
		
		\begin{lem} The map $\alpha = \Jac(\iota)$ above is the map induced by the inclusion $\iota : C \hookrightarrow S$ by the universal property of the Jacobian. More precisely $\alpha$ is given by
			$$ \ko_C(\sum\nolimits_i n_i x_i) \longmapsto \sum\nolimits_i n_i x_i \, .$$
			In particular, $K^d(C)$ is the kernel of this map.
		\end{lem}
		\begin{proof} This follows from the definition of the map $\alpha$, see \refb{albha}. If $F \in \Jac^d(C) \subset M_H(v)$, then $\alpha$ takes on $\Jac^d(C)$ the form  $\ko_C(\sum\nolimits_i n_i x_i) \mapsto \sum\nolimits_i n_i x_i$ which is the map induced by $\iota$ and the universal property of the Jacobian, see subsection \ref{casejacobian}. The second statement is obvious.	\end{proof}
		
		By Lemma \ref{pullinjective} we know that we can see the dual $S^\vee = \Pic^0(S)$ as an abelian subvariety of $\Jac^d(C)$, as the pullback $\iota^\star : \Pic^0(S) \hookrightarrow \Jac(C) \cong \Jac^d(C)$ is an embedding. We conclude that we are in the situation of \ref{casejacobian} and therefore have the following.
		
		\begin{prop} In the situation above, $K^d(C)$ and $S^\vee$ are a pair of complementary abelian subvarieties in the principally polarized abelian variety $\Jac^d(C)$.
		\end{prop}
		\begin{proof} Follows immediately from Lemma \ref{comp}. \end{proof}

		\begin{lem}\cite[Lem.~5.4]{benni1} \label{picard} Let $A$ be an abelian variety. If $\End(A) = \IZ$ then its Picard number is $\rho(A) = 1$. \end{lem} 
		
				We now consider Jacobians of curves which are contained in linear systems defined on abelian surfaces.

				\begin{theo}\label{picarda}\cite[3.B.]{cili} Let $S$ be an abelian surface, $\iota : C \hookrightarrow S$ a smooth curve and let  $$K(C) \coloneqq \ker( \Jac(C) \rightarrow S) \subset \jac(C)$$ be the kernel of the map $\jac(\iota)$ induced by the inclusion and the universal property of the Jacobian, as described in \emph{subsection \ref{casejacobian}}. Then $\End(K(C)) = \IZ$, therefore we have for the Picard number $\rho(K(C)) = 1$.   \end{theo}	
				\begin{proof} We know by Lemma \ref{pullinjective} that $K(C)$ is connected \ie a honest abelian subvariety of $\Jac(C)$. The requirement in \cite[2.II.]{cili} that $|C|$ defines a birational map on its image can be dropped, since the authors only use this to conclude that the map $\iota^{\star} : S^{\vee} \rightarrow \jac(C)$ has finite kernel. In our setting this is the case by Lemma \ref{pullinjective}. 	
					Then by \cite[3.B.]{cili} we have $\End(K(C)) = \IZ$, hence $\rho(K(C)) = 1$ for the Picard number by Lemma \ref{picard}.   	\end{proof}

		Furthermore, we can compute the polarization types of Beauville--Mukai systems of generalized Kummer type.
		
		\begin{theo}\label{prinab} The Picard number of the generic smooth fiber of a Beauville--Mukai system $\pi : X \rightarrow |D|$ of generalized Kummer $n$--type equals one. In particular we have for its polarization type $$\pol(\pi) = (1, \ldots, 1, d_1, d_2) $$
			where $\pol(D) = (d_1, d_2)$ is the type of the polarization defined by $D$. 	
		\end{theo}
		\begin{proof} Let us denote $C \in |D|$ a generic smooth curve. The fiber $F = K(C) = K^d(C)$ of $\pi$ over $C$ is given as the kernel of the map $\Jac(\iota) : \Jac^d(C) \rightarrow S$, see \refb{fiber}, where $\iota : C \hookrightarrow S$ is the inclusion. We are therefore precisely in the situation of Theorem \ref{picarda} which states that $\rho(K(C)) = 1$ for the Picard number. Let $\omega \in \kk_X$ denote a special Kähler class for the fiber $K(C)$. We are in the case of subsection \ref{casejacobian} and by Proposition \ref{kc} the abelian subvariety $K(C)$ admits a polarization $L$ of type $\pol(L) = (1, \ldots, 1, d_1, d_2)$. Since $\rho(K(C)) = 1$, we have $L = \omega|_{K(C)}$ as both are primitive. Therefore $\pol(\pi) = \pol(\omega|_F) = \pol(L) = (1,\ldots,1, d_1, d_2)$ by Proposition \ref{specialpol}. 	\end{proof}
	\end{ueber}

\end{ueber}

\begin{ueber}\textbf{Beauville--Mukai systems in the moduli.} In this section we show that there are Beauville--Mukai systems in each connected component of the moduli of Lagrangian fibrations of $\kdrei$ and generalized Kummer type. We check this in terms of the monodromy invariant.
	
	The proof of the following Proposition is similar to \cite[Ex. 3.1]{eyal2}. However, we give a detailed proof.
	
	\begin{prop}\label{bmkummer} Let $d$ be a positive integer such that $d^2$ divides $n+1$ and let $b$ an integer satisfying $\gcd(d,b) = 1$. Then there exists a Beauville--Mukai system $\pi : K_H(v) \rightarrow \IP^n$ of generalized Kummer type and a primitive isotropic class $\alpha \in H^2(K_H(v), \IZ)$ such that the following holds.
		\begin{enumerate}
			\item $\Div(\alpha) = d$,
			\item the monodromy invariant $\vartheta(\alpha)$ is represented by $(L_{n,d}, (d,b))$,
			\item $c_1(\pi^{\star}\ko_{\IP^n}(1)) = \alpha$. 
			\item Its polarization type is given by $\pol(\pi) = (1, \ldots, 1, d, \frac{n+1}{d})$.
		\end{enumerate}
	\end{prop}
	
	\begin{proof} Let $S$ be an abelian surface together with primitive ample line bundle $L$ on $S$ with $(L,L) = (2n+2)/d^2$.  Set $\beta \coloneqq c_1(L)$ and let $s$ be an integer such that $sb \equiv 1 \mod d$. Then $v \coloneqq (0, d\beta, s)$ is a Mukai vector. In particular $v$ is primitive since $\beta$ is primitive and $\gcd(d,s) = 1$. Choose a $v$--generic ample class $H$. We have $(v,v) = d^2 (\beta,\beta) = 2n+2$ hence $K_H(v) \subset M_H(v)$ is irreducible holomorphic symplectic of dimension $2n$ and we obtain a Beauville--Mukai system $\pi : M_H(v) \rightarrow |L^d|$ as described in section \ref{bmgenkum}. We have Mukai's Hodge isometry
		$$\Theta \ : \ v^{\bot} \longrightarrow H^2(M_H(v), \IZ) \stackrel{r}{\longrightarrow} H^2(K_H(v),\IZ) $$
		see \refb{mukaihom} and \refb{mukaihom2}. The map $r : H^2(M_H(v), \IZ) \longrightarrow H^2(K_H(v),\IZ)$ is the restriction. Recall that the definition of $\Theta$ needs the choice of a quasi--universal family of sheaves $\ke$ on $S$ of similitude $\rho \in \IN$.
		
		Set $\alpha \coloneqq \Theta(0,0,1)$ which is clearly isotropic and define $\iota : H^2(K_H(v),\IZ) \rightarrow H^{\bullet}(S,\IZ)$ to be $\Theta^{-1}$ composed with the inclusion $v^{\bot} \hookrightarrow H^{\bullet}(S,\IZ)$. Note that $\iota$ is a representative of the monodromy invariant orbit constructed in Theorem \ref{orbit}.
		\begin{enumerate}
			\item An element $(r,c,t)$ belongs to $v^{\bot}$ if and only if $$0 \  = \  ((0,d\beta,s), (r,c,t)) \ = \ d(\beta,c) - rs  \ \Longleftrightarrow \ rs \ = \  d(\beta,c) \, .$$
			Hence $d$ divides $r$ since $\gcd(d,s) = 1$. Furthermore, we have $((0,0,1),(r,c,t)) = r$ for all $(r,c,t) \in v^{\bot}$ hence $\Div((0,0,1)) \geq d$. As the lattice of a two torus is $\UU^{\oplus 3}$ \ie in particular unimodular, we have $\Div_{H^2(S,\IZ)}(\beta) = 1$ in $H^2(S,\IZ)$. This implies that $\Div(\beta) = 1$ in $v^{\bot}$, hence we can find an element $c \in H^2(S,\IZ)$ such that $s = (c, \beta)$. Then $(d, c, 0)$ is contained in $v^{\bot}$ and $((0,0,1),(d,c,0)) = d$, hence 
			$$\Div(\alpha) \ = \ \Div(0,0,1) \ = \ d \, .$$
			\item We have $\iota(\alpha) - bv = (0,0,1) - (0,bd\beta, bs) = (0, bd\beta, 1 -bs)$ which is divisible by $d$ since $sb \equiv 1 \mod d$ by assumption. By Lemma \ref{welldefined} (v) the monodromy invariant $\vartheta(\alpha)$ is represented by $(L_{n,d}, (d,b))$.
			
			\item Let $\omega = [p] \in H^4(S,\IZ)$ denote Poincare dual of a point $p \in S$. By our notation we have $\omega = (0,0,1) = \omega^{\vee} \in H^{\bullet}(S)$. Since $S$ is an abelian surface, one has $\sqrt{\td(S)} = 1$, hence $\sqrt{\td(S)} \omega = \omega$. Note that $\ke$ is a sheaf of rank zero, hence $\ch(\ke) = \rho c_1(\ke) + \xi = \rho [D] + \xi$ for some divisor $D$ in $S \times M_H(v)$ and for some terms $\xi$ of higher degree. Furthermore,  $(\pr_{S})^{\star} \omega = [p \times M_H(v)] \in H^4(S \times M_H(v), \IZ)$ and $[(\pr_{M_H(v)})_{!}(\xi \cdot [p \times M_H(v)])]_2 = 0$ due to degree reasons. Then we have 
			\begin{align*} \Theta(0,0,1) \ & = \ r\left( (\pr_{M_H(v)})_{!}\left(D \cdot [p \times M_H(v)]\right) \right) \\ 
			\ & = \ r \left( [F \in M_H(v) \ | \ p \in \supp(F)] \right) \\
			\ & = \ [F \in K_H(v) \ | \ p \in \supp(F)]  \\
			\ & = \ \pi^{\star}[C \in |L^d| \ | \ p \in C] \\
			\ & = \ \pi^{\star}c_1(\ko_{|L^d|}(1)) \ = \ c_1( \pi^{\star} \ko_{|L^d|}(1)) \end{align*} 
			since $V \coloneqq \left\{ C \in |L^d| \ | \ p \in C \right\}$ is a hyperplane in a projective space, hence $[V] = c_1(\ko_{|L^d|}(1))$.  
			\item This follows directly from Theorem \ref{prinab} since $\pol(L) = (1, \frac{n+1}{d^2})$ by Lemma \ref{typeofc} \ie $\pol(dL) = (d, \frac{n+1}{d})$.
		\end{enumerate} \end{proof}
			
\end{ueber}		

\begin{ueber}\textbf{Geometric interpretation of the monodromy invariant.} As in the $\kdrei$--case we have the following \emph{connected component of the moduli of generalized Kummer fibrations}.

	Let $\Lambda$ denote a lattice of signature $(3, b_2 -3)$ which is isometric to the second integral cohomology of an irreducible holomorphic symplectic manifold.
	
	Let $\gM_{\Lambda}$ denote the corresponding moduli space of isomorphism classes of marked pairs $(X, \eta)$ \ie $X$ is an irreducible holomorphic symplectic manifold of the fixed deformation type and $\eta : H^2(X,\IZ) \rightarrow \Lambda$ is a marking. Choose a connected component $\gM^{\circ}_{\Lambda}$ of $\gM_{\Lambda}$ and consider the period map 
	$$ \kp \ : \ \gM^{\circ}_{\Lambda} \longrightarrow \Omega_{\Lambda} \, , \ \ \ (X,\eta) \longmapsto [\eta(H^{2,0}(X))] \, . $$ 
	Choose the orientation of $\tilde{\kc}_{\Lambda}$ compatible to $\gM^{\circ}_{\Lambda}$ in sense of Definition \ref{compatible}. 
	
	Let $\lambda \in \Lambda$ be a nontrivial isotropic class. After a possible change of the sign of $\lambda$ (cf. \ref{foriso}), we have a distinguished and compatible connected component $$\Omega^+_{\lambda^\bot} \ \coloneqq \ \left\{ p \in \Omega_{\lambda^\bot} \ | \ \lambda \in \partial \kc_p \right\}$$ of the hyperplane section $\Omega_{\lambda^{\bot}} = \Omega_{\Lambda} \cap \lambda^{\bot}$, see \ref{foriso}.
	Then define
	$$ \gM^{\circ}_{\lambda^\bot} \ \coloneqq \ \kp^{-1}\left( \Omega^+_{\lambda^\bot} \right) \ = \ \left\{ (X,\eta) \in \gM^{\circ}_{\Lambda} \ | \ \eta^{-1}(\lambda) \text{ is of type } (1,1) \text{ and in  } \partial \kc_X \ \right\} \, .$$
	which is  a connected hypersurface of $\gM^{\circ}_{\Lambda}$ by \cite[Lem.~4.4]{eyal2} and {\cite[Cor.~5.11]{eyalprime}}.
	Consider the nef subspace $$ \gU^{\circ}_{\lambda^\bot} \ \coloneqq \ \left\{ (X,\eta) \in \gM^{\circ}_{\lambda^\bot} \ | \ \eta^{-1}(\lambda) \text{ is nef}  \right\} \, . $$

As in the $\kdrei$--type we have the following result for the generalized Kummer case, which can be proved exactly in the same way as in the $\kdrei$ case with use of \cite[Cor.~1.1]{matiso}.
	
	\begin{theo}\cite[Thm.~3.7]{benni1}\label{modulil} Let $\lambda$ be a primitive and isotropic element in the $\kdrei$ or generalized Kummer lattice. The space $\gU^{\circ}_{\lambda^\bot}$ in the corresponding connected component $\gM^\circ_{\Lambda}$ of the moduli of marked pairs has the following properties. \begin{enumerate} 
			\item It parametrizes isomorphism classes of marked pairs $(X,\eta)$ of $\gM^{\circ}_{\Lambda}$ with $X$ of $\kdrei$ or generalized Kummer type, respectively, admitting a Lagrangian fibration $f : X \rightarrow \IP^n$ such that $$\eta\left(c_1\left(f^{\star} \ko_{\IP^n}(1)\right) \right) \ = \ \lambda \, .$$ 
			\item It is smooth of dimension $20$ for the $\kdrei$ and of dimension $4$ for the generalized Kummer case. Furthermore, it is open in $\gM^{\circ}_{\lambda^\bot}$. 
			\item It is connected. 
		\end{enumerate}\end{theo}

We refer to this space $\gU^\circ_{\lambda^\bot}$ as a \emph{connected component of the moduli of Lagrangian fibrations}.

We can state the geometric interpretation of the monodromy invariant. 

\begin{prop} Let $f_i : X_i \rightarrow \IP^n$, $i = 1,2$, denote two Lagrangian fibrations of generalized Kummer type. Let $\Lambda$ denote the generalized Kummer lattice and set $L_i \coloneqq f_i^{\star}\ko_{\IP^n}(1)$. Then the following statements are equivalent.
	\begin{enumerate}
		\item The Lagrangian fibrations $f_i$ are deformation equivalent.
		\item There exist markings $\eta_i : H^2(X_i, \IZ) \rightarrow \Lambda$ such that the marked pairs $(X_i, \eta_i)$ are contained in the same connected component $\gU^{\circ}_{\lambda^\bot}$ for a primitive isotropic in the generalized Kummer lattice.
		\item \emph{\cite[Lem.~5.17]{eyalprime}} We have $\Div(c_1(L_1)) = \Div(c_1(L_2))$ for the corresponding divisibilities and $\vartheta(c_1(L_1)) = \vartheta(c_1(L_2))$ for the monodromy invariant.
	\end{enumerate}
\end{prop}

\end{ueber}

			\section{Polarization types of generalized Kummer fibrations}\label{computation}
			
			We have now gathered everything to compute the polarization type of Lagrangian fibrations of generalized Kummer type. 
			
			\begin{theo}\label{mainpol} Let $f : X \rightarrow \IP^n$ be a Lagrangian fibration of generalized Kummer type. Then for the polarization type $\pol(f)$ we have 
				$$ \pol(f) \ = \ \left(1, \ldots, 1, d, \frac{n+1}{d}\right) $$
				where $d \coloneqq \Div(c_1(f^\star \ko_{\IP^n}(1)))$ denotes the divisibility of the associated element in the lattice $H^2(X,\IZ)$.
			\end{theo}
			\begin{proof} Let $f : X \rightarrow \IP^n$ denote a Lagrangian fibration of generalized Kummer type and set $L \coloneqq f^{\star} \ko_{\IP^n}(1)$. Then $\lambda \coloneqq c_1(L)$ is primitive and isotropic by Lemma \ref{liso} with respect to the Beauville--Bogomolov quadratic form. Let $d \coloneqq \Div(\lambda)$ denote the divisibility of $\lambda$, note that by Lemma \ref{welldefined} $d^2$ divides $n+1$. Consider the monodromy invariant $\vartheta : I_d(X) \rightarrow \Sigma_{n,d}$ as in Theorem \ref{moninv}. By Lemma \ref{welldefined} (v) there exists an integer $b$ such that $\vartheta(\lambda)$ is represented by $(L_{n,d}, (d,b))$ and we have $\gcd(d,b) = 1$. 
				
				By Theorem \ref{bmkummer} we have a Beauville--Mukai system $\pi : X' \rightarrow \IP^n$ of generalized Kummer type, respectively, together with a primitive isotropic class $\alpha \in H^2(X', \IZ)$ such that $\Div(\alpha) = d$, $L' \coloneqq \pi^{\star} \ko_{\IP^n}(1)$ satisfies $c_1(L') = \alpha$ and $\vartheta(\alpha)$ is represented also by $(L_{n,d}, (d,b))$ \ie $\vartheta(\alpha) = \vartheta(\lambda)$.
				
				By Lemma \ref{isopos} we have $(\omega, L) > 0$ and $(\omega', L') > 0$ for Kähler classes $\omega$ on $X$ and $\omega'$ on $X'$ as $L$ and $L'$ are isotropic and nef, therefore are contained in $\bar{\kk}_X \subset \bar{\kc}_X$ and $\bar{\kk}_{X'} \subset \bar{\kc}_{X'}$ respectively. Hence we can apply Lemma \ref{computeinvariant} to see that the pairs $(X, L)$ and $(X', L')$ are deformation equivalent in the sense of Definition \ref{pairdef}.
				By Proposition \ref{defogen} the Lagrangian fibrations $\pi$ and $f$ are deformation equivalent.
			 By Theorem \ref{mainpolar} and Theorem \ref{bmkummer} we have
				$$ \pol(f) \ = \ \pol(\pi) \ = \  \left(1, \ldots, 1, d, \frac{n+1}{d}\right)$$ 
				which concludes the proof. \end{proof}

\bibliographystyle{alpha}

\end{document}